\documentclass[]{ensmath}

\EnsMath pages 1-2

\usepackage{mathrsfs}
\usepackage{enumerate}
\usepackage{ulem}
\usepackage{fontenc}
\usepackage[all]{xy}
\usepackage{array}

\def\AA {{\mathbb A}}     
\def\CC {{\mathbb C}}     
\def\DD {{\mathbb D}}     
\def\EE {{\mathbb E}}     
\def\HH {{\mathbb H}}     
\def\NN {{\mathbb N}}     
\def\PP {{\mathbb P}}     
\def\QQ {{\mathbb Q}}     
\def\RR {{\mathbb R}}     
\def\TT {{\mathbb T}}     
\def\ZZ {{\mathbb Z}}     

\def\lo  {\longmapsto}

\def\lw  {\longrightarrow}
\def\mc {\mathcal}

\def\ol  {\overline}
\def\pl  {\partial}

\def\rw  {\rightarrow}
\def\tst {\Longleftrightarrow}

\def\wt  {\widetilde}
\def\wh  {\widehat}

\newcommand{\abs}[1]{\left\vert#1\right\vert}

\newcommand{\sco}[1]{\left(#1\right)}
\newcommand{\bsco}[1]{\left\lbrace#1\right\rbrace}

\title[TRIANGLE GROUPS AND TORUS KNOTS]{TRIANGLE GROUPS, AUTOMORPHIC FORMS, AND TORUS KNOTS}

\author[V. V. TSANOV]{Valdemar V. \sn{Tsanov}}

\begin{document}

\maketitle

\begin{abstract} This article is concerned with the relation between several classical and well-known objects: triangle Fuchsian groups, $\mathbb{C}^\times$-equivariant singularities of plane curves, torus knot complements in the 3-sphere. The prototypical example is the modular group $PSL_2(\ZZ)$: the quotient of the nonzero tangent bundle on the upper-half plane by the action of $PSL_2(\ZZ)$ is biholomorphic to the complement of the plane curve $z^3-27w^2=0$. This can be shown using the fact that the algebra of modular forms is doubly generated, by $g_2,g_3$, and the cusp form $\Delta=g_2^3-27g_3^2$ does not vanish on the half-plane. As a byproduct, one finds a diffeomorphism between $PSL_2(\RR)/PSL_2(\ZZ)$ and the complement of the trefoil knot - the local knot of the singular curve. This construction is generalized to include all $(p,q,\infty)$-triangle groups and, respectively, curves of the form $z^q+w^p=0$ and $(p,q)$-torus knots, for $p,q$ co-prime. The general case requires the use of automorphic 
forms on 
the simply connected group $\wt{SL_2}(\RR)$. The proof uses ideas of Milnor and Dolgachev, which they introduced in their studies of the spectra of the algebras of automorphic forms of cocompact triangle groups (and, more generally, uniform lattices). It turns out that the same approach, with some modifications, allows to handle the cuspidal case.
\end{abstract}

\tableofcontents

\section{Introduction}

The motivation for this study came from the author's desire to understand in detail a simple example of a geometric structure on a 3-manifold in the sense of Thurston. The 3-manifold in question is the complement of a torus knot in the 3-sphere and the geometric structure is modelled on $\wt{SL_2}(\RR)$. The existence of such a geometric structure is known from the work of Raymond and Vasquez, \cite{RayVas}. The novelty presented here is an explicit construction using automorphic forms, which, to the author's knowledge, has previously appeared in the literature only for the trefoil knot. This construction has complex analytic flavour, and the result concerning 3-manifolds comes as a byproduct.

Let $p,q$ be a pair of co-prime positive integers and $K_{p,q}$ denote the corresponding torus knot in ${\mathbb S}^3$. Let $\wt{SL_2}(\RR)$ denote the universal covering group of the Lie group $SL_2(\RR)$. In \cite{RayVas}, Raymond and Vasquez have shown that ${\mathbb S}^3\setminus K_{p,q}$ is diffeomorphic to a coset of $\wt{SL_2}(\RR)$ with respect to a suitable discrete subgroup. This result is obtained as a part of a general topological classification of 3-manifolds covered by Lie groups, based on the theory of Seifert fibrations. It turns out that the only nontrivial knot obtained as a coset of the simple group $PSL_2(\RR)$ is the trefoil knot, and the corresponding discrete group is none other than the modular group, so we have ${\mathbb S}^3\setminus K_{2,3}\cong PSL_2(\RR)/PSL_2(\ZZ)$. The latter curious fact has an interesting analytic proof due to Quillen, see \cite{MilnorKTheory} \S 10. Here is a sketch of Quillen's argument.

Let $\HH^2$ denote the upper half-plane with the Poincar${\rm\acute e}$ metric. Recall that the algebra of modular forms for $\Gamma=PSL_2(\ZZ)$ is generated by two elements, often denoted $g_2,g_3$. The modular form $\Delta=g_2^3-27g_3^2$ is the cusp form of lowest positive degree, vanishes with order 1 at $\infty$ and doesn't vanish in the upper half plane $\HH^2$. Let $T'\HH^2$ denote the tangent bundle of $\HH^2$ with removed zero section and let $U\HH^2$ denote the unit tangent bundle. Classically, modular forms are defined as functions on $\HH^2$ with some specific behavior under the action of $\Gamma$. However, it is well known that modular forms can be regarded as specific functions on $T'\HH^2$ invariant on the orbits of (the tangent action of) $\Gamma$. Now define a map $\Psi:T'\HH^2\rw \CC^2$ sending $v$ to $(g_2(v),g_3(v))$. Since the modular forms are constant on $\Gamma$-orbits and $\Delta$ is nonvanishing, $\Psi$ factors through a map $\ol{\Psi}:T'\HH^2/\Gamma\rw \CC^2\setminus V$, where $V=\{
z_1^3-27z_2^2=0\}$. It turns out that $\ol{\Psi}$ is biholomorphic. Now notice that $PSL_2(\RR)$ acts simply transitively on $U\HH^2$, hence the two are diffeomorphic and furthermore $U\HH^2/\Gamma\cong PSL_2(\RR)/\Gamma$. Thus $PSL_2(\RR)/\Gamma$ embeds into $\CC^2\setminus V$. On the other hand, if ${\mathbb S}^3$ is the unit sphere in $\CC^2$, the intersection $K_{2,3}={\mathbb S}^3\cap V$ is a trefoil knot. One can show that $\ol{\Psi}(U\HH^2/\Gamma)$ and ${\mathbb S}^3\setminus K_{2,3}$ are related by an isotopy in $\CC^2$, and this completes the argument.

The present article provides a generalization of this construction incorporating all torus knots. Recall that $K=K_{p,q}$ is obtained as the intersection in $\CC^2$ of ${\mathbb S}^3$ and the algebraic set $V(F)$ of a polynomial $F=c_1z_1^q+c_2z_2^p$ with $c_1,c_2$ nonzero complex numbers, i.e. $K={\mathbb S}^3\cap V(F)$. We intend to find a discrete subgroup $G\subset\wt{SL_2}(\RR)$ whose cusp form of lowest degree matches the polynomial $F$ for suitable $c_1,c_2$. The known structure of torus knot complements suggests that $G$ should be a central extension of a $(p,q,\infty)$-triangle group $\Gamma_{p,q}$. Indeed, let $\Gamma_{p,q}\subset PSL_2(\RR)$ be a $(p,q,\infty)$-triangle Fuchsian group, i.e. the subgroup of orientation preserving elements of a group generated by the reflections on the sides of a geodesic triangle in $\HH^2$ with angles $\pi/p$, $\pi/q$ and $0$ (one cusp). Then $\Gamma_{p,q}$ has a presentation
$$
\Gamma_{p,q} = \; <\alpha_0,\beta_0 \; ;\;\alpha_0^p=1,\beta_0^q=1> \;,
$$
where $\alpha_0$ and $\beta_0$ are elliptic elements in $PSL_2(\RR)$ representing rotations by angles $2\pi/p$ and $2\pi/q$, about the two finite vertices of the triangle, respectively. The full preimage $\wt{\Gamma}\subset\wt{SL_2}(\RR)$ of $\Gamma_{p,q}$ has a presentation
$$
\wt{\Gamma} = \; <\alpha,\beta \; ;\;\alpha^p=\beta^q > \;,
$$
where $\alpha$ and $\beta$ denote suitable preimages of $\alpha_0$ and $\beta_0$, namely those for which $\alpha^p$ is a generator $\zeta$ of the center of $\wt{SL_2}(\RR)$. There is an exact sequence
$$
1 \lw <\zeta> \lw \wt{\Gamma} \lw \Gamma_{p,q} \lw 1 \;.
$$
On the other hand, let us recall the structure of a torus knot complement and the torus knot group $G_K=\pi_1({\mathbb S}^3\setminus K)$. The zero locus of $F$ is invariant under the linear $\CC^\times$-action on $\CC^2$ given by $\lambda(z_1,z_2)=(\lambda^pz_1,\lambda^qz_2)$. The quotient of $\CC^2\setminus\{0\}$ under this action is the weighted projective line $\PP^1(p,q)$. Hence the complement $\CC^2\setminus V(F)$ is a $\CC^\times$-bundle over the orbifold $\Theta=\PP^1(p,q)\setminus\{regular\;point\}$. This orbifold is isomorphic to the quotient $\HH^2/\Gamma_{p,q}$, and hence the orbifold fundamental group $\pi_1(\Theta)$ is represented as $\Gamma_{p,q}$. If $\lambda$ is restricted to the unit circle, the action preserves the unit sphere ${\mathbb S}^3\subset\CC^2$, and defines a Seifert fibration of ${\mathbb S}^3\setminus K$ over $\Theta$. The two singular fibres correspond to generators $A,B$ of $G_K$, while a regular fibre $S$ is expressed as $S=A^p=B^q$ and represents the center of the group. We have a presentation
$$
G_K = <A,B \;;\; A^p=B^q> \;.
$$
Since knot complements are aspherical, the exact sequence of homotopy groups associated with the Seifert fibration has the form
$$
1 \lw <S> \lw G_K \lw \Gamma_{p,q} \lw 1 \;.
$$
In particular, there is an isomorphism $G_K\cong\wt{\Gamma}$. This implies a weak homotopy equivalence $\wt{SL_2}(\RR)/\wt{\Gamma} \sim {\mathbb S}^3\setminus K$. The immediate question is, whether this equivalence is induced by a homeomorphism. It turns out that the answer is ``yes'' only for the trefoil knot. Notice that, for any natural number $r$ co-prime to both $p$ and $q$, the elements $\alpha_r=\alpha^r$ and $\beta_r=\beta^r$ generate a subgroup $G_r\subset\wt{\Gamma}$ which is normal, of index $r$, and abstractly isomorphic to $\wt\Gamma$. There is a presentation
$$
G_r = \; <\alpha_r,\beta_r \; ;\;\alpha_r^p=\beta_r^q > \;.
$$
Thus $G_r\cong G_K$, and $\wt{SL_2}(\RR)/G_r$ is weakly homotopy equivalent to ${\mathbb S}^3\setminus K$. Now, our task is to determine for which $r$ a diffeomorphism occurs. An interesting result, presented in this paper, is that the requested $r$ is naturally ``chosen'' by the automorphic forms of $\wt{\Gamma}$. To achieve this we need to consider the appropriate automorphic forms for discrete subgroups of $\wt{SL_2}(\RR)$, i.e. we allow fractional degrees and characters. To this end, we follow the approach of Milnor, who has determined in \cite{MilnorBrieskornMani} the algebras of automorphic forms for centrally extended co-compact triangle groups.\\

The main results are contained in sections \ref{Sect Auto-f Gamma_pq}, \ref{Sect SL/G = S-K}, \ref{Sect A Lens Space}, and are briefly summarized as follows: The algebra of automorphic forms for $\wt{\Gamma}$ is generated by two forms $\omega_a,\omega_b$, viewed as functions on the universal cover $\wt{T'\HH^2}$ of $T'\HH^2$. The cusp form of lowest positive degree is expressed as $\omega_\infty=c_a\omega_a^p+c_b\omega_b^q$ and does not vanish anywhere in $\wt{T'\HH^2}$. This allows us to define a map $\Psi:\wt{T'\HH^2}\rw\CC^2$, which factors through $\ol{\Psi}:\wt{T'\HH^2}/G \lw \CC^2\setminus V(F)$, where $F=c_az_1^q+c_bz_2^p$ and $G\subset\wt{\Gamma}$ is the common kernel of the characters of $\wt{\Gamma}$ corresponding to the generators $\omega_a$ and $\omega_b$. We have $G=G_{pq-p-q}$. There is a diffeomorphism $\wt{SL_2}(\RR)/G \cong {\mathbb S}^3\setminus K$ obtained from the map $\ol{\Psi}$. Finally, the homogeneous space $PSL_2(\RR)/\Gamma_{p,q}$ is found to be diffeomorphic to a knot complement in a 
suitable lens space.\\

The paper contains a fair amount of known material included in order to reduce the necessary prerequisites to a minimum. Besides $\cite{MilnorBrieskornMani}$, another very important reference is Ogg's book \cite{Ogg}. Ogg works in the classical setting where the automorphic forms are defined as functions on $\HH^2$, fractional degrees and characters are allowed but $\wt{SL_2}(\RR)$ is not mentioned explicitly. The first chapter of \cite{Ogg} contains a treatment of $(2,q,\infty)$-triangle groups, and a construction for the generators of the algebra of automorphic forms. The approaches in \cite{MilnorBrieskornMani} and \cite{Ogg} are combined to obtain generators of the algebra of automorphic forms for $\wt{\Gamma}$.

Other sources presenting similar perspectives on the relation between automorphic forms and quasi-homogeneous singularities include the following. Dolgachev, \cite{Dolga1974}, gives an outline of the relation just mentioned, in the case of co-compact Fuchsian groups, with an emphasis on special triangle groups. The same author, in \cite{Dolga1975}, outlines a generalization of these ideas, considering uniform lattices acting on higher dimensional complex homogeneous space. Wagreich, \cite{Wagreich1}, \cite{Wagreich2}, \cite{Wagreich3}, considers Fuchsian groups of the first kind, not necessarily co-compact. In particular, Wagreich \cite{Wagreich2} provides a classification of all such groups whose algebra of automorphic forms admits a generating set of at most 3 elements. This result does not cover our case because Wagreich considers only forms of integral degree, while here we allow fractional degrees. Perhaps it would be interesting to try to carry out Wagreich's program form \cite{Wagreich2}, but allowing 
fractional degrees and characters. More recently, Natanzon and Pratoussevitch, \cite{NatanPratou}, have used and developed this framework in their study of moduli spaces of Gorenstein singularities.

Very recently, Pinsky \cite{Pinsky} has obtained results related to the ones presented here, using completely different methods. In particular, she has given an alternative explicit construction of the diffeomorphism from corollary \ref{Coro L-K=PSL/Gamma}. Pinsky's work appeared only after this paper was posted on arxiv.org. The author would like to thank the referee for this remark.

\section{Background}\label{Sect_Background}

\subsection{Geometric structures}\label{Sect_GeoSrt}

A complete and locally homogenous Riemannian metric $g$ on a manifold $M$ is called a {\it geometric structure} on $M$. Here complete means that
every geodesic in $M$ may be extended to $(-\infty,\infty)$; locally homogenous means that any two points in $M$ have isometric neighborhoods. A complete, homogenous, simply connected Riemannian manifold $(N, h)$, whose isometry group is maximal, is called a {\it geometry}. We shall say that the geometric structure $(M,g)$ is modelled on the geometry $(N,h)$, if each point of $M$ has a neighborhood, isometric to an open set of $N$.

Let $M$ be a manifold and let $\wt{M}$ be the universal cover of $M$. If $\wt{M}$ admits a  geometry $g$ such that the covering action of $\pi_1(M)$ is by isometries, the covering map induces a geometric structure on $M$ modelled on $(\wt{M},g)$.

In dimension 2, the theory initiated by Poincar\'e and Klein culminates in the uniformization theorem proven independently by Poincar$\acute{\rm e}$ and Koebe. It states, that there are three geometries: the Euclidian plane, the hyperbolic plane and the sphere:
$$
\EE^2 \; , \quad \HH^2 \; , \quad {\mathbb S}^2 \;.
$$
Any 2-manifold has a geometric structure obtained by identifying its universal cover with one of the three geometries above. Some 2-manifolds ($\RR^2, \quad \RR^2\setminus \{point\}$, M$\ddot{\rm o}$bius band) have geometric structures modelled on two distinct geometries.

The situation with 3-manifolds is much more complicated. Thurston, \cite{ThurstonBull}, showed that any 3-dimensional geometric structure is to be modelled on one of the following eight geometries:
\begin{gather}\label{ScottSpaces}
\begin{array}{lll}
\EE^3 \qquad \qquad \qquad  & {\mathbb S}^2 \times \EE \qquad \qquad \qquad & {\rm Nil} \\
\HH^3 & \HH^2 \times \EE & {\rm Sol} \\
{\mathbb S}^3 & \widetilde{SL_2}(\RR) & \\
\end{array}
\end{gather}
However, it is not true that any 3-manifold readily admits a geometric structure, the situation is far more complicated than in the case of 2-manifolds. The theory was led to a great success culminating with the work of Perelman, \cite{Perel1}, \cite{Perel2}, \cite{Perel3}. See e.g. \cite{TianMorg} for an exposition and further references. We shall refrain from comments on the general results, as our concerns here are modest in this respect. We only consider concrete, well-known manifolds and we are interested in explicit construction, rather than abstract existence, of geometric structures. We refer the reader to the surveys of Scott \cite{Scott} and Bonahon \cite{Bonahon}, which predate the general existence theorem, but contain excellent descriptions of the eight geometries, as well as many examples. The definitions and results needed for our purposes are stated in the text below.

\subsection{Seifert fibrations and orbifolds}\label{Sect_Seifert_and_Orbi}

In this section we discuss briefly a class of 3-manifolds introduced by Seifert \cite{Seifert}. These are manifolds admitting a so called Seifert fibration - a special kind of circle foliation. The "base" has an orbifold structure as described below. The Seifert fibration structures on 3-manifolds are closely related to their geometric properties, as explained in Scott's article \cite{Scott}, which we follow here to some extend.

We start with the general definition of an orbifold as introduced by Thurston in \cite{Thurston}. Intuitively, an orbifold is a space locally modelled on $\RR^n$ modulo finite group actions.

\begin{definition}\label{Def Orbifold}
An n-dimensional {\it orbifold} $\Theta$ consists of the following data. A Hausdorff paracompact topological space $X_{\Theta}$, called the {\it underlying space}, covered by a collection $\{U_{i}\}$ of open sets, called {\it charts}, closed under finite intersections. To each $U_i$ is associated a finite group $\Gamma_i$, an action of $\Gamma_i$ on an open subset $\tilde{U}_i$ of $\RR^n$, and a homeomorphism $\varphi_i:U_i \lw \tilde{U}_i/\Gamma_i$. Whenever $U_i \subset U_j$ there is an injective homomorphism
\begin{gather*}
f_{ij}:\Gamma_i \hookrightarrow \Gamma_j
\end{gather*}
and an embedding
\begin{gather*}
\tilde{\varphi}_{ij}:\tilde{U}_i \hookrightarrow \tilde{U}_j
\end{gather*}
equivariant with respect to $f_{ij}$ (i.e., for $\gamma \in \Gamma_i$, $\tilde{\varphi}_{ij}(\gamma x)=f_{ij}(\gamma)\tilde{\varphi}_{ij}(x)$) such that the diagram below commutes:
\begin{gather*}
\begin{array}{ccc}
 \tilde{U}_i & \stackrel{\tilde{\varphi}_{ij}}{\lw} & \tilde{U}_j \\
 \Big\downarrow & & \Big\downarrow \\
 \tilde{U}_i/\Gamma_i & \stackrel{\varphi_{ij}=\tilde{\varphi}_{ij}/\Gamma_i}{\lw} &
 \tilde{U}_j/\Gamma_j \\
 \Big\uparrow & & \quad \Big\downarrow f_{ij} \\
 \varphi_i \; \vert \; \quad \quad & & \tilde{U}_j/\Gamma_j \\
 \Big\vert & & \Big\uparrow \\
 U_i & \hookrightarrow & U_j
\end{array}
\end{gather*}
If $x\in X_{\Theta}$ and $U=\tilde{U}/\Gamma$ is a chart about $x$, we denote by $\Gamma_x$ the isotropy group of any point in the preimage of $x$ in $\tilde{U}$. The set $\Sigma_{\Theta}:=\{x \in X_{\Theta} : \Gamma_x \ne \{1\} \}$ is called the {\it singular locus} of $\Theta$. The points of $\Sigma_{\Theta}$ are called {\it singular}, the rest of the points of $X_\Theta$ - {\it regular}.
\end{definition}

Note, that the maps $\tilde{\varphi}_{ij}$ are defined up to composition with elements of $\Gamma_j$, and $f_{ij}$ are defined up to conjugation by elements of $\Gamma_j$. It is not generally true that $\tilde{\varphi}_{ik}=\tilde{\varphi}_{jk}\circ\tilde{\varphi}_{ij}$ when $U_i \subset U_j \subset U_k$, but there should exist an element $\gamma \in \Gamma_k$ such that $\gamma \tilde{\varphi}_{ik}=\tilde{\varphi}_{jk}\circ \tilde{\varphi}_{ij}$ and $\gamma f_{ik}(g) \gamma^{-1} = f_{jk} \circ f_{ij}(g)$. Another remark to be made is that the covering $\{U_i\}$ is not an intrinsic part of the structure of the orbifold: two coverings give rise to the same orbifold structure if they can be combined consistently to give a finer cover still satisfying the above conditions.

Note also, that any manifold is an orbifold with empty singular locus.

\begin{definition}\label{Def OrbiCover} An {\it orbifold-cover} of an orbifold $\Theta$ is an orbifold $\Xi$ together with a projection $\rho : X_{\Xi} \lw X_{\Theta}$, such that every point $x \in X_{\Theta}$ has a neighborhood $U=\tilde{U}/\Gamma$ such that each component $V_i$ of $\rho^{-1}(U)$ is isomorphic to $\tilde{U}/\Gamma_i$, where $\Gamma_i$ is a subgroup of $\Gamma$. The isomorphism $\rho^{-1}(U)\cong\tilde{U}/\Gamma_i$ is required to respect the projections. We use the notation $\rho:\Xi \lw \Theta$.

An orbifold-cover $\varrho:\wt{\Theta}\lw \Theta$ is called {\it universal}, if for any other orbifold-cover $\rho:\Xi \lw \Theta$, there is a lift $\tilde{\varrho}:\wt{\Theta}\lw \Xi$ which is an orbifold-cover, and $\varrho = \rho\circ\tilde{\varrho}$.

An orbifold is called {\it good} if it admits an orbifold-cover which is in fact a manifold, and {\it bad} - otherwise.
\end{definition}

\begin{proposition}\label{Prop ManiCoverOrbi}{\rm (Thurston \cite{Thurston})} Let $M$ be a manifold and let $G$ be a group acting properly discontinuously on $M$. Then $M/G$ has the structure of an orbifold and the projection $M\lw M/G$ is an orbifold-cover.
\end{proposition}

In the above notation, if $H$ is a normal subgroup of $G$, then $M/H$ is an orbifold-cover of $M/G$ under the action of the factorgroup $G/H$. If $M$ is a simply connected manifold, it is the universal orbifold-cover of $M/G$. Existence of a universal orbifold-cover $\wt{\Theta}$ for an arbitrary orbifold ${\it \Theta}$ is shown in \cite{Thurston}, along with the description of the corresponding group of deck transformations called the {\it orbifold fundamental group} and denoted $\pi_1(\Theta)$.

In the 2-dimensional case, both the universal cover and the fundamental group of an orbifold are easier to describe. An outline can be found in \cite{Scott}. If $\Theta$ is a good 2-dimensional orbifold, covered by a manifold $M$, then we can define a geometric structure on $\Theta$ using the geometry of the universal cover $\wt{M}$ of $M$, which must be $\EE^2, {\mathbb S}^2$ or $\HH^2$.

From now on we consider only 2-dimensional orbifolds and we restrict the type of allowed singularities to cone points, which are defined as follows. A singular point $x\in X_\Theta$ is called a {\it cone point} of index $k$, if there is a chart $U$ about $x$ such that the corresponding group $\Gamma_x$ is isomorphic to $\ZZ_k$ and acts by rotations around a point $\tilde{x} \in \tilde{U}$. Note that if $\Theta$ is an orbifold whose singular locus contains only cone points, then the singularities of any orbifold-cover of $\Theta$ are also limited to cone points. Scott, in \cite{Scott}, gives the following method to compute the fundamental group of a 2-dimensional orbifold with cone points.

\begin{proposition}\label{Prop Orbigroup Scott}
Let $\Theta$ be a 2-dimensional orbifold with singular locus consisting of a finite number of cone points $a_1,...,a_n$ with indexes $k_1,...,k_n$ respectively. Let $D_1,...,D_n$ be disjoint disc neighborhoods of $a_1,...,a_n$. Set $N:=X_{\Theta}\setminus (D_1\cup...\cup D_n)$. Let $H$ be the smallest normal subgroup of $\pi_1(N)$ containing the elements $c_1^{k_1},...,c_n^{k_n}$, where $c_j$ represents the circle $\pl D_j$. Then
\begin{gather*}
\pi_1(\Theta) \cong \pi_1(N)/H
\end{gather*}
Suppose $\pi_1(N)$ has a presentation with generators $x_1,...,x_l$ and relations $r_1=1,...,r_m=1$, where $r_j$ is some word in the generators. Then $\pi_1(\Theta)$ has a presentation
$$
\pi_1(\Theta)=<x_1,...,x_l\; ;\; r_1=1,...,r_m=1,c_1^{k_1}=1,...,c_n^{k_n}=1> \;.
$$
\end{proposition}

We are now ready to proceed with the definition of Seifert fibred 3-manifolds, and to describe some of their basic properties.

\begin{definition}\label{Def fibred tori}
The solid torus $T := \DD^2 \times {\mathbb S}^1$ is a trivial circle bundle over the disc $\DD^2:=\{ z \in \CC : \abs{z}=1\}$. The fibres are $(z,e^{2\pi i t}), \; t\in [0,1]$. The torus $T$ with this fibration structure is called {\it trivially fibred torus}.

Let $p$ and $q$ be co-prime integers. The solid torus $\DD^2 \times {\mathbb S}^1$ with the circle foliation
\begin{gather*}
(ze^{2\pi i \frac{p}{q}t},e^{2\pi i t}) \quad , \quad t \in [0,q]
\end{gather*}
is called {\it (p,q)-twisted fibred torus} and denoted by $T(p,q)$. The central fibre $\{0\} \times {\mathbb S}^1$ of a fibred torus is called the {\it core}.
\end{definition}

\begin{definition}\label{Def Seifert Fibr} A 3-manifold $M$ is called {\it Seifert manifold}, if it can be presented as a disjoint union of circles, called fibres, satisfying the following property. Each fibre $l$ admits a tubular neighborhood $T_l$ in $M$, consisting of fibres, such that $T_l$ is a fibred torus (possibly trivial) with core $l$.

If a given fibre possesses a neighborhood which is a trivially fibred torus, the fibre is called {\it regular}. Otherwise the fibre is called {\it singular}.

A fibration of this kind is called a {\it Seifert fibration} on $M$. The fibred tori $T_l$ are called {\it trivializing tori} of the fibration.
\end{definition}

A fibred torus $T(p,q)$ can be covered by a trivially fibred torus in different ways. Each of the following transformations generates an action of the
cyclic group $\ZZ_q$ on $\DD^2 \times {\mathbb S}^1$, which sends fibres to fibres.
\begin{gather}\label{R_p/q on DxS 1}
\begin{array}{ccc}
\DD^2 \times {\mathbb S}^1 & \lw & \DD^2 \times {\mathbb S}^1 \\
(z,e^{2\pi i t}) & \lo & (ze^{2\pi i \frac{p}{q}t},e^{2\pi i t})
\end{array}
\end{gather}
\begin{gather}\label{R_p/q on DxS 1/q}
\begin{array}{cccc}
\DD^2 \times {\mathbb S}^1 & \lw & \DD^2 \times {\mathbb S}^1 \\
(z,e^{2\pi i t}) & \lo & (ze^{2\pi i \frac{p}{q}t},e^{2\pi i (t+\frac{1}{q})})
\end{array}
\end{gather}
The first transformation keeps all points in the core fixed, and its action results in an orbifold-covering between the solid tori. The second one defines a regular covering of manifolds.

The set of fibres in a fibred torus form an orbifold described as follows. In the notation of the definition of a fibred torus, put $D=\DD^2\times \{1\}\subset T$. If $T$ is trivially fibred, then each fibre intersects $D$ exactly once, and hence $D$ parametrizes the fibres. In this case we have a regular fibration, and the base space is the manifold $D$. If $T$ is $(p,q)$-fibred, then the core intersects $D$ once, while each of the other fibres intersects $D$ at $q$ distinct points. The set of fibres is the orbifold $D/\ZZ_q$, where $\ZZ_q$ is viewed as the groups of $q$-th roots of 1 acting on $D$ by multiplication; the singular locus consists of a single cone point. More generally, the set of fibres in a Seifert fibred 3-manifold M is a 2-dimensional orbifold $\Theta$ whose singular locus consists of cone points corresponding to the singular fibres in $M$. In such a case, $\Theta$ is called the {\it base orbifold} of the Seifert fibration $M$. When the singular fibres are finite in number, we can compute 
the 
fundamental group of the base orbifold using Proposition \ref{Prop Orbigroup Scott}.

\begin{remark}\label{Zab SubgrupRegFibre}
Let $M$ be a connected Seifert fibred manifold.

(*) All regular fibres in $M$ are isotopic. Indeed, let $l$ and $l'$ be two regular fibres. Let $\gamma$ be a path connecting $l$ with $l'$, which does not intersect a singular fibre. Then $\gamma$ can be covered by a finite number of open trivially fibred solid tori $T_1,...,T_n$, such that $T_1 \supset l$ and $T_n \supset l'$. Now the result follows from the fact, that all fibres in a trivially fibred solid torus are isotopic.

(**) Let $l$ be a regular fibre, let $m$ be a base point for $\pi_1(M)$ lying on a regular fibre, and let $\gamma$ be a path from the point $m$ to $l$. Then the element $\tau \in \pi_1(M,m)$ represented by $\gamma^{-1} . l . \gamma$ does not depend on the choice of $\gamma$. Indeed, if $\gamma_1$ is another path from $m$ to $l$ and $l_m$ be the regular fibre containing $m$, then, by (*) above, $\gamma^{-1} . l . \gamma \sim l_m \sim \gamma_1^{-1} . l . \gamma_1$, where $\sim$ denotes a homotopy keeping the base point $m$ fixed.

Hence, the subgroup of $\pi_1(M,m)$ generated by $\tau$ is normal, and does not depend on the choice of $m$ and $l$. We refer to this subgroup as {\it the subgroup of $\pi_1(M)$ generated by a regular fibre}.
\end{remark}

\begin{lemma}\label{Lemma ExactSeq SeifertFib}
Let $M$ be a connected Seifert fibred manifold with base orbifold $\Theta$. Let $L$ be the subgroup of $\pi_1(M)$ generated by a regular fibre. Then $L$ is contained in the center of $\pi_1(M)$ and there is an exact sequence
\begin{gather}\label{ExactSeq SeifertFib (Lema)}
1\lw L \lw \pi_1(M) \lw \pi_1(\Theta) \lw 1
\end{gather}
\end{lemma}

\begin{proof} See e.g. \cite{Scott} Lemma 3.2.
\end{proof}

\begin{example}
Let $\Theta$ be a good 2-dimensional orbifold and let $P$ be its universal cover. Thus $P$ is either ${\mathbb S}^2,\,\EE^2$ or $\HH^2$, and $\Theta$ is isomorphic to $P/\Gamma$ for some discrete subgroup $\Gamma\subset {\rm Isom}(P)$. Let $UP$ be the unit tangent bundle on $P$, i.e. the bundle whose total space consists of all tangent vectors to $P$ with unit length. Then $\Gamma$ acts freely on $UP$ and the quotient $M=UP/\Gamma$ is a smooth manifold. There is a Seifert fibration on $M$ whose fibres are the images of the fibres $U_pP$, for $p\in P$. The base orbifold of this Seifert fibration is $\Theta$. It is natural to call $M$ the unit tangent bundle on $\Theta$. The fundamental group $\pi_1(M)$ is a central extension of $\Gamma$ by $\pi_1(UP)$. For the three possible cases we have $\pi_1(U{\mathbb S}^2)=\ZZ_2$, $\pi_1(U\EE^2)=\ZZ$, $\pi_1(U\HH^2)=\ZZ$.
\end{example}

\subsection{Torus knots}\label{Sect Knots}

We refer to Burde and Zieschang \cite{BurdeZieschang} for the classical knot theoretic notions, some history and bibliography. In this section we focus on some properties of torus knots needed for our purposes.

An embedded circle $K \hookrightarrow {\mathbb S}^3$ is called a {\it knot}. The knot type of $K$ is the isotopy equivalence class of the embedding. We only consider tame (smooth) knots; for such knots $K$ and $L$ the relation of isotopy is equivalent to the existence of an orientation preserving homeomorphism of ${\mathbb S}^3$ sending $K$ to $L$. The 3-sphere will be considered as the unit sphere in the complex plane $\CC^2$, i.e.
\begin{gather}\label{S^3 < C^2}
{\mathbb S}^3 = \{ (z_1,z_2)\in \CC^2 : \abs{z_1}^2+\abs{z_2}^2=1 \} \;.
\end{gather}
A knot is called {\it trivial} (or {\it unknotted}), if it is equivalent to the circle $U:=\{\abs{z_1}=1,z_2=0\}$. A nontrivial knot is called a {\it torus knot}, if it is equivalent to a simple closed curve on the torus
\begin{gather}\label{T^2 < S^3}
T^2:=\{ (z_1,z_2)\in \CC^2 : \abs{z_1}=a,\abs{z_2}=b\} \; \subset \; {\mathbb S}^3
\end{gather}
where $a$ and $b$ are positive numbers satisfying $a^2+b^2=1$. The isotopy class of such a curve $K$ is determined by the pair of co-prime integers $(p,q)$, which represent $K$ as an element of the fundamental group of $T^2$ in terms of the standard generators $\{\abs{z_1}=a,z_2=b\}$ and $\{z_1=a,\abs{z_2}=b\}$. In such a case $K$ is called a $(p,q)$-{\it torus knot} and is denoted by $K_{p,q}$. We have the relation $K_{p,q}\sim K_{p',q'}$ if and only if $(p',q')$ equals $(-p,-q),(q,p)$ or $(-q,-p)$. The knot $K_{p,-q}$ is the mirror image of $K_{p,q}$; the complements ${\mathbb S}^3 \setminus K_{p,q}$ and ${\mathbb S}^3 \setminus K_{p,-q}$ are homeomorphic but not under an orientation preserving homeomorphism.

The fundamental group $\pi_1({\mathbb S}^3\setminus K)$ is called the {\it knot group} and is denoted by $G_K$. One important result characterizing torus knots, conjectured by Neuwirth and proved by Burde and Zieschang, is the following: a knot group admits a non-trivial center if and only if the knot is either an unknot or a torus knot; see \cite{BurdeZieschang} for details and further references. This algebraic characterization is related to the presence of a Seifert fibration in a torus knot complement, the classical result of Seifert mentioned earlier.

Let us be more explicit. Let $p$ and $q$ be co-prime natural numbers, fixed for the rest of the paper. Assume $p<q$. Consider the linear $\CC^\times$ action on $\CC^2$ given by
\begin{gather*}
\begin{array}{ccc}
\CC^{\times} \times \CC^2 & \lw & \CC^2 \\
\lambda\cdot(z_1,z_2)&=&(\lambda^pz_1,\lambda^qz_2) \;.
\end{array}
\end{gather*}
The origin ${\bf o}:=(0,0)\in \CC^2$ is a fixed point of the action. The quotient
$$
\PP^1(p,q) = (\CC^2\setminus\{{\bf o}\})/\CC^\times
$$
is the so called (p,q)-weighted complex projective line. Restricting $\lambda$ to ${\mathbb S}^1$ we obtain the unitary flow
\begin{gather}\label{potoka}
h_t := \begin{pmatrix} e^{2\pi ipt} & 0 \\
                       0 & e^{2\pi iqt}
       \end{pmatrix}
        \quad , \quad t\in \RR \;.
\end{gather}

\begin{proposition}\label{Prop S(p,q)=S^3/h_t}
The action of the flow $h_t$ defines a Seifert fibration in ${\mathbb S}^3$. The base orbifold, denoted $\PP^1(p,q)$, has underlying topological space ${\mathbb S}^2$, and two singular points: cone points of indices $p$ and $q$ respectively. The orbifold fundamental group is trivial.
\end{proposition}

\begin{proof}
All 3-spheres centered in ${\bf o}$ are invariant under $h_t$. Thus $h_t$ defines a circle foliation in ${\mathbb S}^3$. To show that this foliation is actually a Seifert fibration we need to find a system of trivializing tori (see definition \ref{Def Seifert Fibr}). The coordinate complex lines $\CC_1$ and $\CC_2$ in $\CC^2$ are also $h_t$-invariant. $\CC_1\setminus \{ {\bf o}\}$ and $\CC_2\setminus \{ {\bf o}\}$ consist of the orbits with periods $\frac{1}{p}$ and $\frac{1}{q}$ respectively. The remaining orbits have period 1. Put
\begin{align*}
S_1 = & \CC_1\cap{\mathbb S}^3=\{ (z_1,z_2)\in \CC^2 : \abs{z_1}=1,\; z_2=0\} \;, \\
S_2 = & \CC_2\cap{\mathbb S}^3=\{ (z_1,z_2)\in \CC^2 : z_1=0,\; \abs{z_2}=1 \} \;.
\end{align*}
The solid torus ${\mathbb S}^3\setminus S_1$ is $h_t$-invariant. The orbits of the flow are of the form $(z_1e^{2\pi i pt},z_2e^{2\pi i qt})$, and by reparametrization $t'=tq$ we obtain $(z_1e^{2\pi i \frac{p}{q}t'},z_2e^{2\pi i t'})$. Hence, the flow $h_t$ defines a structure of a (p,q)-twisted fibred torus in ${\mathbb S}^3\setminus S_1$. The core of this fibred torus is $S_2$. It is fixed by the transformation $h_{\frac{1}{q}}$. Analogously, ${\mathbb S}^3\setminus S_2$ is a (q,p)-twisted fibred torus, whose core $S_1$ is fixed by the transformation $h_{\frac{1}{p}}$. All orbits except $S_1$ and $S_2$ have period 1, and are regular in the sense of Seifert. Thus ${\mathbb S}^3$ is a Seifert fibration with two singular fibres $S_1$ and $S_2$. The base orbifold $\PP^1(p,q)$ has two singular cone points $a$ and $b$ of indexes $p$ and $q$ respectively. The set of regular points $\PP^1(p,q)\setminus\{a,b\}$ is an annulus, because it is the orbit space of ${\mathbb S}^3\setminus(S_1\cup S_2)$ which is a fibred 
torus with removed core. It follows that the underlying surface of $S(p,q)$ is a 2-sphere. The fact that the orbifold fundamental group of $S(p,q)$ is trivial can be obtained directly from Proposition \ref{Prop Orbigroup Scott} using the fact that $p$ and $q$ are co-prime. Alternatively, it follows from Lemma \ref{Lemma ExactSeq SeifertFib} and the fact that ${\mathbb S}^3$ is simply connected.
\end{proof}

Any regular orbit of $h_t$ is a $(p,q)$-torus knot in the 3-sphere in which it belongs. Explicitly, let $z_{01},z_{02}$ be two nonzero complex numbers, then the orbit through $(z_{01},z_{02})$ is a $(p,q)$-knot lying on the torus
\begin{gather*}
\{ (z_1,z_2)\in \CC^2 : \abs{z_1}=\abs{z_{01}},\; \abs{z_2}=\abs{z_{02}} \}\subset {\mathbb S}^3_{|z_{01}|^2+|z_{02}|^2} \;.
\end{gather*}

Now let $c_1,c_2$ be any two nonzero complex numbers, and consider the polynomial
\begin{gather*}
f(z_1,z_2):=c_1z_1^q+c_2z_2^p \;.
\end{gather*}
The analytic set $V(f) := \{ (z_1,z_2)\in \CC^2 : c_1z_1^q+c_2z_2^p=0 \}$ has a unique singularity at the origin. We have
\begin{gather}\label{flowActFunc}
f(\lambda\cdot(z_1,z_2))=\lambda^{pq}f(z_1,z_2) \;.
\end{gather}
Hence, $V(f)$ is invariant under the action of $\CC^\times$. It does not intersect the coordinate lines $\CC_1$ and $\CC_2$, except at the point ${\bf o}$, and so it does not contain singular orbits of $h_t$. The intersection
$$
K=V(f)\cap{\mathbb S}^3
$$
is a regular orbit of $h_t$, and hence a the (p,q)-torus knot in ${\mathbb S}^3$.

To simplify the notation, we take $c_1=c_2=1$. Then $K$ can be expressed as
\begin{gather}\label{Vyzela}
K = \; \left\{ \begin{array}{l} z_1=a_1 \; e^{2\pi ipt} \\
                     z_2=a_2 \; e^{2\pi iqt + \frac{i\pi}{p}} \;
       \end{array}
       \quad \vert \quad t\in[0,1] \; \right\} 
\end{gather}
with
\begin{gather*}
a_1^2+a_2^2=1 \quad,\quad a_1^q-a_2^p=0 \quad , \quad a_1,a_2 > 0 \;.
\end{gather*}
$K$ lies on the torus $T:=\{ \abs{z_1}=a_1 ,\abs{z_2}=a_2 \}$. Since $K$ is itself one regular orbit of $h_t$, its complement ${\mathbb S}^3 \setminus K$ is $h_t$-invariant. As a direct consequence of Proposition \ref{Prop S(p,q)=S^3/h_t} we obtain

\begin{proposition}\label{Prop SeifFib TorKnot}
The flow $h_t$ defines a Seifert fibration in ${\mathbb S}^3\setminus K$. The base orbifold, $\Theta$, is a 2-dimensional orbifold with underlying space $\RR^2$ (a punctured 2-sphere) and two singular points: cone points of indexes $p$ and $q$.
\end{proposition}

We have a decomposition
\begin{gather*}
{\mathbb S}^3\setminus K = \TT(p,q) \sqcup (T\setminus K) \sqcup \TT(q,p) \;,
\end{gather*}
where
\begin{gather}\label{The TrivTori of S-K}
\begin{array}{l}
S_1 \hookrightarrow \TT(p,q):=\{ \abs{z_1}<a_1 \; , \; \abs{z_2}=\sqrt{1-\abs{z_1}^2} \} \;, \\
\; \\
K \hookrightarrow T:=\{ \abs{z_1}=a_1 \; , \; \abs{z_2}=a_2 \} \;, \\
\; \\
S_2 \hookrightarrow \TT(q,p):=\{ \abs{z_1}=\sqrt{1-\abs{z_2}^2} \; ,\; \abs{z_2}<\mu \} \;. \\
\end{array}
\end{gather}
From this decomposition of ${\mathbb S}^3\setminus K$ we can produce the well-known presentation of the torus knot group:
\begin{gather}\label{TheKnotGroup}
G_{p,q}:=G_K=\pi_1({\mathbb S}^3\setminus K_{p,q}) \cong <S_1,S_2\; ;\; S_1^p=S_2^q> \;.
\end{gather}
The element $S=S_1^p=S_2^q$ represents a regular fibre. This element is a generator of $\pi_1(T\setminus K)$ and of the center $Z(G_K)$ (see remark \ref{Zab SubgrupRegFibre} and Lemma \ref{Lemma ExactSeq SeifertFib}). We have the central extension
$$
1 \lw <S> \lw G_K \lw \pi_1(\Theta) \lw 1 \;.
$$
Thus $\pi_1(\Theta)\cong G_K/Z(G_K) \cong \ZZ_p*\ZZ_q$. The universal orbifold-cover of $\Theta$ will be constructed in section \ref{Sect Def Triang Groups}, see Proposition \ref{Prop H->Theta}.

\section{Automorphic forms on $\wt{SL_2}(\RR)$}\label{Sect SL geom}

The Lie group $\wt{SL_2}(\RR)$ is defined as the universal covering group of the simple Lie group $PSL_2(\RR)$. In this section we describe the group $\wt{SL_2}(\RR)$ by identifying it with the universal cover of the unit tangent bundle $U\HH^2$ on the hyperbolic plane. We start by identifying $PSL_2(\RR)$ with $U\HH^2$. Thus $\wt{SL_2}(\RR)$ is endowed with the structure of a topologically trivial $\RR$-fibration over the hyperbolic plane. This construction can be used to obtain a left invariant metric on $\wt{SL_2}(\RR)$ which makes it a 3-dimensional geometry, one of the eight geometries of Thurston's list. Furthermore, this allows us to view automorphic forms for discrete groups acting on the upper half plane as functions on $\wt{SL_2}(\RR)$. The fact that we pass to the universal cover allows the consideration of forms of fractional degrees which we need for our main construction.

\subsection{Isometric group actions on the hyperbolic plane}

We work with the upper half plane model of the hyperbolic plane
\begin{gather*}
\HH^2 = \{ z \in \CC : {\rm Im}(z)>0\} \; ,\quad
ds^2=\frac{\abs{dz}^2}{{\rm Im}(z)^2} \;.
\end{gather*}
The matrix Lie group
\begin{gather*}
SL_2(\RR):= \bsco{ A =
       \begin{pmatrix} a & b \\
                       c & d
       \end{pmatrix} \; :\; a,b,c,d \in \RR,\; ad-bc=1 }
\end{gather*}
acts isometrically on $\HH^2$ by M$\ddot{\rm o}$bius transformations:
\begin{gather*}
A.z :=\frac{az+b}{cz+d} \;.
\end{gather*}
This action is transitive. The stabilizer of the point $i\in \HH^2$ is $SO_2(\RR)$, isomorphic to the circle group ${\mathbb S}^1$. The center of $SL_2(\RR)$ consists of the elements $I,-I$, and acts trivially on $\HH^2$. To make the action effective, we take the quotient
$$
P: SL_2(\RR) \lw PSL_2(\RR)\cong SL_2(\RR)/\{\pm I\} \;.
$$
The action of $PSL_2(\RR)$ is effective and transitive. The stabilizer of the point $i$ is $PSO_2(\RR)$, also isomorphic to ${\mathbb S}^1$. The coset space $PSL_2(\RR)/PSO_2(\RR)$ is diffeomorphic to the hyperbolic plane, the stabilizer being simply transitive on the circle $U_i\HH^2$ of unit tangent vectors at the point $i$. Thus $PSL_2(\RR)$ is diffeomorphic to the unit tangent bundle on $\HH^2$:
\begin{gather*}
U\HH^2 := \{ (z,\vec{v}) : z\in \HH^2,\; \vec{v} \in T_z\HH^2 \; ,\; \norm{\vec{v}}_{\HH}=1 \} \;.
\end{gather*}
To fix a diffeomorphism we choose $(i,\frac{\pl}{\pl z}_{\vert i})\in U\HH^2$ to be a base point, and identify $PSL_2(\RR)$ with its orbit through that point. Since the hyperbolic plane is homeomorphic to a 2-cell, the fibre bundles $T\HH^2$ and $U\HH^2$ are topologically trivial. Hence $U\HH^2$ is diffeomorphic to $\HH^2 \times {\mathbb S}^1$ and the fundamental group $\pi_1(U\HH^2)\cong \pi_1(PSL_2(\RR))$ is an infinite cyclic group.

The Lie group $\wt{SL_2}(\RR)$ is defined as the universal covering group of $PSL_2(\RR)$. We denote the canonical projection by $\wt{P}:\wt{SL_2}(\RR)\lw PSL_2(\RR)$. The fundamental group of $PSL_2(\RR)$ is identified with the center $C$ of $\wt{SL_2}(\RR)$. Thus we have a central extension
\begin{gather}\label{For C->wt(SL)->PSL}
1 \lw C \lw \wt{SL_2}(\RR) \stackrel{\wt{P}}{\lw} PSL_2(\RR) \lw 1 \;.
\end{gather}
The group $\wt{SL_2}(\RR)$ acts, via $\wt{P}$, on $\HH^2$ and $U\HH^2$. Since $PSL_2(\RR)$ is identified with $U\HH^2$, it follows that $\wt{SL_2}(\RR)$ is identified with the universal covering $\wt{U\HH^2}$ (which is in turn diffeomorphic to $\HH^2\times\RR$). This endows $\wt{SL_2}(\RR)$ with the structure of a topologically trivial $\RR$-bundle over the hyperbolic plane. We fix the generator $c$ of $C$ corresponding to a simultaneous counter-clockwise rotation of all unit tangent vectors to $\HH^2$ by angle $2\pi$ keeping the base points fixed.

For any integer $r\geq 1$ we can consider the group $\wt{SL_2}(\RR)/<c^r>$, which is a $r$-fold covering of $PSL_2(\RR)$. Clearly $\wt{SL_2}(\RR)/<c^r>$ has cyclic center of order $r$ and is diffeomorphic to $\HH^2\times{\mathbb S}^1$.

\begin{remark}\label{Zab SL_2 Geometry}
The hyperbolic metric on $\HH^2$ can be used to induce a left invariant metric on $PSL_2(\RR)$ and from there on $\wt{SL_2}(\RR)$. This turns $\wt{SL_2}(\RR)$ into a model for a geometry, in the sense discussed in section \ref{Sect_GeoSrt}. The exact form of this metric is not important for the purposes of this paper, so we omit the precise formulation.
\end{remark}

\subsection{Automorphic forms for discrete subgroups of $\wt{SL_2}(\RR)$}\label{Sect Auto-forms}

If $G$ is any subgroup of $\wt{SL}_2(\RR)$, then $G\cap C \subset \; Z(G)$ and $\wt{P}(G)\cong G/(G\cap C)$. Clearly there are discrete subgroups of $\wt{SL_2}(\RR)$ which project to non-discrete subgroups of $PSL_2(\RR)$. On the other hand, if $\Gamma$ is a discrete subgroup of $PSL_2(\RR)$, i.e. a Fuchsian group, then its preimage $\wt{\Gamma}=\wt{P}^{-1}(\Gamma)$ is discrete in $\wt{SL_2}(\RR)$. We are mainly interested in discrete subgroups $\wt{\Gamma}$ of $\wt{SL_2}(\RR)$ arising as preimages of Fuchsian groups and, more generally, subgroups $G\subset\wt{\Gamma}$ which are obviously discrete as well, and project to Fuchsian groups.

Let $G$ be a discrete subgroup of $\wt{SL_2}(\RR)$. We assume that the projection $\wt{P}(G)\subset PSL_2(\RR)$ is a Fuchsian group of the first kind, i.e. has a fundamental domain with finite volume. In such a case, the quotient $\HH^2/\wt{P}(G)$ can be compactified (holomorphically) by adding a finite number of points. This section is devoted to the notion of an automorphic form for $G$. We consider forms of rational degrees as defined by Milnor in \cite{MilnorBrieskornMani}; this approach is somewhat non-standard, but suitable for our purposes.

Let $\wt{\CC^\times}$ denote the universal covering group of the multiplicative group $\CC^\times=\CC\setminus\{0\}$. As a complex Lie group $\wt{\CC^\times}$ is isomorphic to the additive group $\CC$, but we prefer to keep the multiplicative notation. Recall that for $\zeta\in\wt{\CC^\times}$ we have a well-defined $r$-th root $\zeta^{\frac{1}{r}}$ for any positive integer $r$, and hence we have a well-defined $k$-th power $\zeta^k$ for any rational number $k$.

Let $T'\HH^2$ denote the tangent bundle on the hyperbolic plane with removed zero-section. This bundle has a nonvanishing holomorphic section $\frac{\pl}{\pl z}$. Hence $T'\HH^2$ is holomorphically trivial, i.e. biholomorphic to $\HH^2\times\CC^\times$. Notice that $T'\HH^2$ contains the unit tangent bundle $U\HH^2$ considered earlier. The isomorphism $T'\HH^2\cong\HH^2\times\CC^\times$ associated with the section $\frac{\pl}{\pl z}$ is such that $U_z\HH^2$ is identified with $\{z\}\times\{\zeta\in\CC^\times:\abs{\zeta}={\rm Im}(z)\}$. The universal cover $\wt{T'\HH^2}$ is biholomorphic to $\HH^2\times\wt{\CC^\times}$. The group $\wt{SL_2}(\RR)$ acts on $\wt{T'\HH^2}$, and hence on $\HH^2\times\wt{\CC^\times}$, in an obvious way and all stabilizers are trivial. Thus any orbit of this action is diffeomorphic to $\wt{SL_2}(\RR)$. We fix the embedding $\wt{SL_2}(\RR) \subset \HH^2\times\wt{\CC^\times}$ given by the orbit map through the point $(i,1)\in\HH^2\times\wt{\CC^\times}$.

Let $\gamma\in\wt{SL_2}(\RR)$. The expression for the transformation of $\HH^2\times\wt{\CC^\times}$ associated with $\gamma$ is
$$
\gamma(z,w)=(\gamma(z),\gamma'(z)w) \;,
$$
where $\gamma(z)$ is given by the usual action of $\wt{SL_2}(\RR)$ on $\HH^2$ and $\gamma':\HH^2\lw\wt{\CC^\times}$ is the lift of the (nonvanishing) derivative $\frac{d\gamma(z)}{dz}$ determined by the requirement to satisfy the chain rule $(\gamma_2\gamma_1)'(z)=\gamma_2'(\gamma_1(z))\gamma_1'(z)$.

\begin{definition}
A (differential) {\it form} of degree $k\in\QQ$, or a $k$-form, on $\HH^2$ is defined to be a complex valued function $\omega$ of two variables $z\in \HH^2$ and $dz \in \wt{\CC^\times}$ of the form
$$
\omega(z,dz)=f(z)dz^k \;,
$$
where $f$ is a holomorphic function on $\HH^2$. The product of $f(z)$ and $dz^k$ is taken after projecting $dz^k$ from $\wt{\CC^\times}$ to $\CC^\times$ via the universal covering map.
\end{definition}

Let $\omega=f(z)dz^k$ be a form on $\HH^2$. For any $\gamma\in\wt{SL_2}(\RR)$ we have the pullback
$$
\gamma^*\omega(z,dz) = \omega(\gamma.(z,dz)) = f(\gamma(z))(\gamma'(z))^k dz^k \;.
$$

\begin{definition} Let $\omega(z,dz)=f(z)dz^k$ be a form on $\HH^2$. If $\chi\in{\rm Hom}(G,U(1))$ is a character of $G$, a form $\omega$ on $\HH^2$ is called $\chi$-{\it automorphic}, if the following two conditions are satisfied:

{\rm (i)} $\gamma^*\omega=\chi(\gamma)\omega$ for all $\gamma\in G$.

{\rm (ii)} Suppose $\wt{P}(G)$ has a fundamental domain with a cusp at $\infty$ (the case of a different cusp is treated by sending it to $\infty$). Let $z\lo z+\lambda$ be a generator of the parabolic subgroup of $G$ fixing $\infty$. Then the function $f(z)$ is required to have an expansion of the form
$$
f(z)=\sum\limits_{n=0}^\infty a_n e^{2\pi i n \frac{z}{\lambda}}
$$
which is convergent for $z\in\HH^2$. In such a case, $f$ is said to be holomorphic at $\infty$.

If $\chi$ is the trivial character of $G$, a $\chi$-automorphic form will be called $G$-{\it automorphic}; in such a case, we have $\gamma^*\omega=\omega$ for all $\gamma\in G$. When the group is understood from the context, by automorphic form we will mean a form which is $\chi$-automorphic for some character $\chi$. An automorphic form vanishing at a cusp of $\wt{P}(G)$ is called a {\it cusp form}.
\end{definition}

The automorphy property of a form $\omega(z,dz)=f(z)dz^k$ is reflected on a property of the function $f$. Namely,
$$
\gamma^*\omega=\chi(\gamma)\omega \quad \tst \quad f(\gamma(z))(\gamma'(z))^{k}=\chi(\gamma)f(z) \;.
$$

The space of $\chi$-automorphic $k$-forms will be denoted by ${\mc A}_G^{k,\chi}$. The algebra of forms on $\HH^2$ generated by all automorphic forms for $G$ will be denoted by ${\mc A}_G^{*,*}$. This is a bi-graded algebra with components ${\mc A}_G^{k,\chi}$. The $G$-automorphic forms generate a subalgebra, to be denoted by ${\mc A}_G^{*}$, with degree components ${\mc A}_G^{k}={\mc A}_G^{k,1}$. Since we have embedded $\wt{SL_2}(\RR)$ into $\HH^2\times\wt{\CC^\times}$, any form $\omega$ on $\HH^2$ can be viewed, via restriction, as a complex valued function on $\wt{SL_2}(\RR)$. The elements of ${\mc A}_G^*$ are then well-defined functions on the coset space $\wt{SL_2}(\RR)/G$.

\begin{lemma}\label{Lemma AutForm and Levels}
Let $r,s$ be co-prime integers, $r>0$, and let $\omega(z,dz)=f(z)dz^{\frac{s}{r}}$ be a form of degree $\frac{s}{r}$ on $\HH^2$.

{\rm (i)} The form $\omega$ is invariant under the action of the central subgroup $<c^r>\subset \wt{SL_2}(\RR)$, and hence is a well-defined function on the group $\wt{SL_2}(\RR)/<c^r>$.

{\rm (ii)} If $\omega$ is $G$-automorphic, then $G\cap C \subset <c^r>$.
\end{lemma}

\begin{proof}
The generator $c$ of the center of $\wt{SL_2}(\RR)$ satisfies $c(z)=z$ and $c'(z)^k$ projects to $e^{2\pi i k}\in\CC^\times$ for all $z\in\HH^2$ and $k\in\QQ$. Hence, for $n\in\NN$, we have
$$
(c^n)^*\omega(z,dz) = f(c^n(z))c'(z)^{n\frac{s}{r}}dz^{\frac{s}{r}} = f(z)e^{2\pi i \frac{ns}{r}}dz^\frac{s}{r}=e^{2\pi i \frac{ns}{r}}\omega(z,dz) \;.
$$
Now, since $s$ and $r$ are assumed co-prime, we see that $(c^n)^*\omega=\omega$ if and only if $r$ divides $n$. Both parts of the lemma follow immediately.
\end{proof}

The following lemma is taken form \cite{MilnorBrieskornMani}; it describes roots of automorphic forms.

\begin{lemma}\label{Lemma Roots of forms}
Let $\omega(z,dz)=f(z)dz^k$ be a $\chi$-automorphic form. Suppose that $f$ possesses an $n$-root, i.e. there is a holomorphic function $f_1$ on $\HH^2$ such that $f_1(z)^n=f(z)$. Then the form $\omega_1(z,dz)=f_1(z)dz^{\frac{k}{n}}$ is $\chi_1$-automorphic for some character $\chi_1$ of $G$ satisfying $\chi_1^n=\chi$.
\end{lemma}

\begin{proof}
Let $\gamma\in G$. Then $\omega_1$ and $\gamma^*\omega_1$ are both forms of degree $\frac{k}{n}$ on $\HH^2$. Hence the quotient $\frac{\gamma^*\omega_1}{\omega_1}$ is a well-defined meromorphic function on $\HH^2$. The $n$-th power of this function is $\frac{\gamma^*\omega}{\omega}=\chi(\gamma)$ which is constant. Therefore $\frac{\gamma^*\omega_1}{\omega_1}$ is constant as well. Define $\chi_1(\gamma)=\frac{\gamma^*\omega_1}{\omega_1}$. Then $\chi_1$ is a character of $G$ and we have $\chi_1^n=\chi$.
\end{proof}

\section{Triangle groups and torus knot groups}\label{Sect TriGr & TKGr}

\subsection{Triangle groups}\label{Sect Def Triang Groups}

In this section, we discuss a beautiful class or Fuchsian groups - the triangle groups. They have been extensively studied, a basic introduction may be found in Beardon's book \cite{Beardon}.

For any two distinct points $a$ and $b$ in $\HH^2$ denote by $\ol{ab}$ the directed geodesic segment connecting $a$ and $b$. The same notation will be used for the limit case, when one or both points lie on the absolute. Consider a geodesic triangle $\Delta$ in $\HH^2$, with vertices $a,b,c$ (some of them may lie on the absolute). To fix the notation, suppose that the triple $a,b,c$ is positively oriented. Suppose that the interior angles of $\Delta$, at the vertices $a,b,c$ respectively, are equal to $\frac{\pi}{p}, \frac{\pi}{q}, \frac{\pi}{r}$, with $p,q,r \in \NN\cup\{\infty\}$ \footnote{We assume that $\frac{\pi}{\infty}=0$ for a vertex on the absolute. To obtain a hyperbolic triangle, one must have $\frac{1}{p}+\frac{1}{q}+\frac{1}{r}<1$. This is certainly true when $1<p<q,\;r=\infty$.}. Denote by $\sigma_a,\sigma_b,\sigma_c$ the reflections in $\HH^2$ with respect to $\ol{bc},\ol{ac},\ol{ab}$. The triangle $\Delta$ is a fundamental polygon for a discrete group of isometries of $\HH^2$, with 
presentation
\begin{align*}
\Gamma^*(p,q,r)\cong \; <\sigma_a,\sigma_b,\sigma_c\; ;\;&
\sigma_a^2=\sigma_b^2=\sigma_c^2=1 \;, \\
& (\sigma_b\sigma_c)^p= (\sigma_c\sigma_a)^q=(\sigma_a\sigma_b)^r=1>
\end{align*}
Such a group is called a $(p,q,r)$-{\it triangle group}. The elements $\alpha_0:=\sigma_b\sigma_c$, $\beta_0:=\sigma_c\sigma_a$, $\gamma_0:=\sigma_a\sigma_b$ are elliptic or parabolic \footnote{In the case when some of the parameters $p,q,r$ equal $\infty$, for instance $r=\infty$, the corresponding element $\gamma_0=\sigma_a\sigma_b$ is parabolic. The relation $\gamma_0^r=1$ is omitted.} elements of $PSL_2(\RR)$, with fixed points $a,b,c$ and angles of rotation $\frac{2\pi}{p},\frac{2\pi}{q},\frac{2\pi}{r}$ respectively. We have $\gamma_0=\beta_0^{-1}\alpha_0^{-1}$.

If $\Delta_1$ is another triangle in $\HH^2$, similar (same angles) to $\Delta$, then there exists an isometry of $\HH^2$ sending $\Delta$ to $\Delta_1$. Thus any two $(p,q,r)$-triangle groups are conjugate in ${\rm Isom}(\HH^2)$. The generators of $\Gamma^*(p,q,r)$ do not preserve the orientation in $\HH^2$, so this group is not Fuchsian.

To obtain a Fuchsian group we take the subgroup of index 2 of $\Gamma^*(p,q,r)$ consisting of all orientation preserving elements. We denote this group by $\Gamma(p,q,r)$ and, following the tradition, also refer to it as a $(p,q,r)$-{\it triangle group}; the distinction should be easily made from the context. We are mostly concerned with the orientation preserving triangle groups. To obtain a fundamental domain for $\Gamma(p,q,r)$, denote $\Delta':=\sigma_c(\Delta),\; d:=\sigma_c(c)$, and set
\begin{gather*}
D:=\Delta \cup \Delta' \;.
\end{gather*}
Then $D$ is a quadrangle in $\HH^2$ with vertices $a,d,b,c$ and angles $\frac{2\pi}{p},\frac{\pi}{r},\frac{2\pi}{q},\frac{\pi}{r}$ respectively. The edges of $D$ are identified by $\alpha_0$ and $\beta_0$ as follows
\begin{gather}\label{Action Gue H}
\begin{array}{l}
\alpha_0:\ol{ad}\lo\ol{ac} \\
\beta_0:\ol{bc}\lo\ol{bd} \;.
\end{array}
\end{gather}
According to Poincar${\rm \acute{e}}$'s theorem for fundamental polygons, $\alpha_0$ and $\beta_0$ generate a Fuchsian group with fundamental domain $D$ and presentation
\begin{gather}\label{Quad group p,r,q,r}
\Gamma(p,q,r)\cong \; <\alpha_0,\beta_0 \; ;\;\alpha_0^p=1,\beta_0^q=1,(\alpha_0\beta_0)^r=1> \;.
\end{gather}
It is easy to see that $\sigma_b(\Delta)=\alpha_0(\Delta')$ and $\sigma_a(\Delta)=\beta_0^{-1}(\Delta')$, thus
$D_1:=\Delta\cup\sigma_b(\Delta)$ and $D_2:=\Delta\cup\sigma_a(\Delta)$ are alternative fundamental domains for $\Gamma(p,q,r)$. The independence on the choice of reflection generating the factorgroup is obvious from the definition, and also in view of the presentation
\begin{gather}\label{Present Gamma(p,r,q,r)}
\Gamma(p,q,r)\cong \; <\alpha_0,\beta_0,\gamma_0 \; ;\; \alpha_0^p=1,\beta_0^q=1,\gamma_0^r=1,\alpha_0\beta_0\gamma_0=1> \;.
\end{gather}

A classical example is the modular group $PSL_2(\ZZ)$. Its traditional fundamental domain and generators are
\begin{gather*}
D_1(PSL_2(\ZZ))=\bsco{x+iy\in \HH^2 \; :\; -\frac{1}{2}<x<\frac{1}{2}\; ,\; y>\sqrt{1-x^2}} \;, \\
\; \\
P(S),\; P(T) \;\; ,\;\;{\rm where}\;\;
S=\begin{pmatrix} 0 & -1 \\
                  1 & 0
  \end{pmatrix} \quad, \quad
T=\begin{pmatrix} 1 & 1 \\
                  0 & 1
  \end{pmatrix} \;.
\end{gather*}
The relations are $S^2=1$ and $(ST^{-1})^3=1$. This group is Fuchsian, it preserves the orientation of $\HH^2$. In the notation introduced above, $PSL_2(\ZZ)\cong \Gamma(2,3,\infty)$ is a triangle group. It is a subgroup of index 2 of the triangle group $\Gamma^*(2,3,\infty)$ with fundamental domain
\begin{gather*}
\Delta(PSL_2(\ZZ))=\bsco{x+iy\in \HH^2\; :\; 0<x<\frac{1}{2}\; ,\; y>\sqrt{1-x^2}} \;.
\end{gather*}
Alternative fundamental domain and generators for $PSL_2(\ZZ)$ are
\begin{gather*}
D(PSL_2(\ZZ))=\bsco{x+iy\in \HH^2 \; :\; 0<x<\frac{1}{2}\; , \;
y>\sqrt{1-(x-1)^2}} \;, \\
\; \\
P(S) , \; P(U) \;\;,\;\;{\rm where}\;\; U=ST^{-1} \;.
\end{gather*}
Hence, if we denote $\alpha_0:=P(S),\; \beta_0:=P(U)$, we have
\begin{gather*}
PSL_2(\ZZ) \cong \; <\alpha_0,\beta_0 \; ;\; \alpha_0^2=1,\beta_0^3=1> \; \cong \Gamma(2,3,\infty) \;.
\end{gather*}

We now focus on the class of triangle groups which actually concern us for the rest of the paper. Let $p$ and $q$ be co-prime numbers, and set
\begin{gather}\label{Quad group p,inf,q,inf}
\Gamma_{p,q}:=\Gamma(p,q,\infty) \cong \; <\alpha_0,\beta_0\; ;\; \alpha_0^p=1,\beta_0^q=1> \;.
\end{gather}
Abstractly, $\Gamma_{p,q}$ is isomorphic to the free product $\ZZ_p * \ZZ_q$. The triangle $\Delta$ has exactly one vertex, say $c$, lying on the absolute. The fundamental domain $D$ corresponding to the above presentation is a quadrangle with two cusps $c$ and $d$, and two vertices $a$ and $b$ inside $\HH^2$ with angles $\frac{2\pi}{p}$ and $\frac{2\pi}{q}$ respectively. We can take
\begin{gather}\label{Javni formuli Gamma(p,q)}
\begin{array}{c}
a=-\cos \frac{\pi}{p}+i\sin \frac{\pi}{p}, \quad d=0,\quad
b=\cos \frac{\pi}{q}+i\sin \frac{\pi}{q}, \quad c=\infty \\
\; \\
D = \left\{ x+iy \in \HH^2 \; :\quad
\begin{matrix}
-\cos \frac{\pi}{p}<x<\cos \frac{\pi}{q} \qquad \qquad \qquad \\
\; \\
y>\sqrt{\sco{\frac{1}{2\cos\frac{\pi}{p}}}^2-\sco{x+\frac{1}{2\cos\frac{\pi}{p}}}^2} \\
y>\sqrt{\sco{\frac{1}{2\cos\frac{\pi}{q}}}^2-\sco{x-\frac{1}{2\cos\frac{\pi}{q}}}^2}
\end{matrix}
\quad\right\} \;, \\
\; \\
\alpha_0=P(A_0) ,\; \beta_0=P(B_0) ,\quad
A_0=\begin{pmatrix} 2\cos \frac{\pi}{p} & 1 \\
                    -1 & 0
    \end{pmatrix} , \quad
B_0=\begin{pmatrix} 0 & 1 \\
                    -1 & 2\cos \frac{\pi}{q}
    \end{pmatrix} \;.
\end{array}
\end{gather}
The action of the generators is given in (\ref{Action Gue H}). The two cusps are identified by the generators.

\begin{remark}\label{zab AltFundDom Gamma(p,q) D_1}
An alternative fundamental domain with one cusp for $\Gamma_{p,q}$
is
\begin{gather*}
D_1 \; := \; \Delta \cup \sigma_b(\Delta) \; = \; (D \cap
\alpha_0(D))\cup \alpha_0(D)
\end{gather*}
It corresponds to the following generators and relations
\begin{gather*}
\Gamma_{p,q}\cong\; <\alpha_0,\gamma_0 \; ;\; \alpha_0^p=1 \; , \; (\gamma_0\alpha_0)^q=1 > \;.
\end{gather*}
The element $\gamma_0=(\alpha_0\beta_0)^{-1}$ is parabolic with fixed point $\infty$.
\end{remark}

The quotient $\HH^2/\Gamma_{p,q}$ is an orbifold with underlying topological space ${\mathbb S}^2\setminus\{point\}$. The missing point corresponds to the cusp of the fundamental domain $D_1$. This orbifold is an important ingredient in our construction, it was already encountered as the base orbifold $\Theta$ of the Seifert fibration in a torus knot complement, see Proposition \ref{Prop SeifFib TorKnot}.

\begin{proposition}\label{Prop H->Theta}
The projection $\HH^2\lw\HH^2/\Gamma_{p,q}$ is the universal covering of the orbifold $\Theta$, i.e. $\Theta\cong\HH^2/\Gamma_{p,q}$.
\end{proposition}

\begin{proof}
The generators of $\Gamma_{p,q}$ give the following identifications of edges of $D$:
\begin{gather*}
\alpha_0:\ol{ad}\lo\ol{ac} \quad , \quad \beta_0:\ol{bc}\lo\ol{bd} \;.
\end{gather*}
The diagonal $\ol{cd}$ cuts $D$ in two triangles, say $D_a$ and $D_b$, containing the vertices $a$ and $b$ respectively. These triangles are mapped under the action of $\alpha_0$ and $\beta_0$ on two cones $D^2/\ZZ_p$ and $D^2/\ZZ_q$. These two cones are the base orbifolds of the fibred tori $\TT(q,p)$ and $\TT(p,q)$ given in (\ref{The TrivTori of S-K}). The projection of the annulus $T\setminus K$, from formula (\ref{The TrivTori of S-K}), is the diagonal $\ol{cd}$. The cusps $c$ and $d$ are identified by both $\alpha_0$ and $\beta_0$, and correspond to the knot $K$.
\end{proof}

\subsection{The commutator subgroup of a $(p,q,\infty)$-triangle group}

Let $\Gamma_{p,q}'$ denote the commutator subgroup of $\Gamma_{p,q}$. The main result in this section states that $\Gamma_{p,q}'$ is a free group of rank $(p-1)(q-1)$. This result is somewhat tangential to the main line of the paper, and in fact can be deduced easily from well known properties of torus knot groups, cf \cite{BurdeZieschang}. However, the proof we give here is somewhat interesting in that it uses the geometry of triangle groups. In particular we construct fundamental domain for the action of $\Gamma_{p,q}'$ on $\HH^2$ and give free generators. The freeness of the commutator subgroup of a torus knot group can be deduced from this result.

Since $\Gamma_{p,q}$ is isomorphic to the free product $\ZZ_p*\ZZ_q$, it follows that the factorgroup $\Gamma_{p,q}/\Gamma_{p,q}'$ is isomorphic to the direct product $\ZZ_p\times\ZZ_q$. Since $p$ and $q$ are co-prime, we can conclude that $\Gamma_{p,q}/\Gamma_{p,q}'$ is cyclic of order $pq$. For $\gamma\in\Gamma_{p,q}$ let $\ol{\gamma}\in\Gamma_{p,q}/\Gamma_{p,q}'$ denote the image of the canonical projection.

\begin{remark}\label{Zab Gamma/Gamma' = Zpq}

{\rm (i)} If $\Gamma_1$ is a normal subgroup of $\Gamma_{p,q}$ such that $\Gamma_{p,q}/\Gamma_1$ is cyclic of order $pq$, then $\Gamma_1$ coincides with the commutator $\Gamma_{p,q}'$.

{\rm (ii)} The factorgroup $\Gamma_{p,q}/\Gamma_{p,q}'$ consists of the elements $\ol{\alpha}_0^i\ol{\beta}_0^j$, $i=0,...,p-1$, $j=0,...,q-1$.

{\rm (iii)} Each of the two elements below generates $\Gamma_{p,q}/\Gamma_{p,q}'$
\begin{gather*}
\ol{\gamma}_0=(\ol{\alpha}_0\ol{\beta}_0)^{-1} \quad,\quad \ol{x}_0=\ol{\alpha}_0^{q_1}\ol{\beta}_0^{p_1} \;,
\end{gather*}
where $p_1,q_1$ are integers such that $pp_1+qq_1=1$. We have $\ol{\gamma}_0^{-pp_1^2-qq_1^2}=\ol{x}_0$ and $\ol{x}_0^{pq-p-q}=\ol{\gamma}_0$.
\end{remark}

\begin{proposition}\label{Prop Comut Quad Group}
The commutator subgroup $\Gamma_{p,q}'$ is free of rank $(p-1)(q-1)$. Any of the following two sets generates $\Gamma_{p,q}'$ freely.
\begin{gather}\label{SetGen Gamma' xi}
\xi_{0(i,j)} := [\alpha_{0}^i,\beta_{0}^j] =
\alpha_0^i\beta_0^j\alpha_0^{-i}\beta_0^{-j} \quad , \;
i=1,...,p-1,\; j=1,...,q-1 \;, \qquad
\end{gather}
\begin{gather}\label{SetGen Gamma' eta}
\quad \; \eta_{0(k,j)} :=
[\alpha_0,\alpha_0^k\beta_0^j\alpha_0^{-k}] =
\alpha_0^k[\alpha_0,\beta_0^j]\alpha_0^{-k} \; , \;
k=0,...,p-2,\; j=1,...,q-1 .
\end{gather}
\end{proposition}

\begin{proof}
The commutator $\Gamma_{p,q}'$ is the smallest normal subgroup of $\Gamma_{p,q}$ containing $[\alpha_0,\beta_0]$. Denote by $\Xi$ the subgroup of $\Gamma_{p,q}$ generated by the elements $\xi_{0(i,j)}$. The inclusion $\Xi \subset \Gamma_{p,q}'$ is clear. To prove that $\Gamma_{p,q}' \subset \Xi$, we will show that $\Xi$ is a normal subgroup of $\Gamma_{p,q}$. It is sufficient to see that the conjugates of $\xi_{0(i,j)}$ by $\alpha_0$ and $\beta_0$ are still in $\Xi$. Direct computations show that
\begin{gather}\label{F-la alpha^-1 xi alpha}
\alpha_0 \xi_{0(i,j)}\alpha_0^{-1} = \xi_{0(i+1,j)}\xi_{0(1,j)}^{-1} \quad,\quad
\beta_0 \xi_{0(i,j)}\beta_0^{-1} = \xi_{0(i,1)}^{-1}\xi_{0(i,j+1)} \;.
\end{gather}
Hence, $\Xi$ is a normal in $\Gamma_{p,q}$ and $\Xi \equiv \Gamma_{p,q}'$.

The elements (\ref{SetGen Gamma' xi}) are expressed in terms of (\ref{SetGen Gamma' eta}) in the following way. First we have
\begin{gather*}
\eta_{0(0,j)}=\xi_{0(1,j)} \quad , \quad j=1,...,q-1 \;.
\end{gather*}
Using (\ref{F-la alpha^-1 xi alpha}) and induction on $i$, it is easy to show that
\begin{gather*}
\eta_{0(i-1,j)}\eta_{0(i-2,j)}...\eta_{0(1,j)}\eta_{0(0,j)}=\xi_{0(i,j)}
\quad ,\quad i=1,...,p-1,\; j=1,...,q-1 \;.
\end{gather*}
Hence, the elements $\eta_{0(i,j)}$ also generate $\Gamma_{p,q}'$.

To obtain a fundamental domain for $\Gamma_{p,q}'$ one may take the union of the images of the fundamental domain $D$ of $\Gamma_{p,q}$ under the action of the elements representing the factorgroup $\Gamma_{p,q}/\Gamma_{p,q}'$. We denote
\begin{gather}\label{FundDomainF A}
D' := \bigcup\limits_{i,j} D_{i,j} \quad ,\;\;{\rm where}\;\; D_{i,j} := \alpha_0^i\beta_0^j(D) \; ,\;\; i=0,...,p-1,\;
j=0,...,q-1 \;.
\end{gather}
$D'$ is a polygon in $\HH^2$ with $2p(q-1)$ vertices, of which $p(q-1)$ cusps and $p(q-1)$ points inside $\HH^2$, changing alternatively. The interior angles at the vertices inside $\HH^2$ are all equal to $\frac{2\pi}{p}$. The edges of $D'$ are
\begin{gather*}
\begin{array}{ll}
l_{i,j}' := \alpha_0^i\beta_0^j(\ol{ac}) & \quad , \quad i=0,...,p-1 \;, \\
l_{i,j}'' := \alpha_0^i\beta_0^j(\ol{ad}) & \quad ,\quad j=1,...,q-1 \;.
\end{array}
\end{gather*}
The arcs $\alpha_0^i(\ol{ad})$ and $\alpha_0^i(\ol{ac})$ lie in the interior of $D'$. A direct computation shows that
\begin{gather*}
\eta_{0(i,j)} : l_{i,j}' \lo l_{i+1,j}'' \quad,\quad i \in \ZZ(mod \; p) \;,\; j = 1,...,q-1 \;.
\end{gather*}
According to Poincar$\acute{\rm e}$'s theorem, $D'$ is a fundamental polygon for a Fuchsian group with generators
\begin{gather*}
\eta_{0(i,j)} \quad ,\quad i=0,...,p-1,\; j=1,...,q-1
\end{gather*}
and relations
\begin{gather*}
\eta_{0(p-1,j)}\eta_{0(p-2,j)}...\eta_{0(1,j)}\eta_{0(0,j)}=1 \quad,\quad j=1,...,q-1 \;.
\end{gather*}
These relations correspond to the cycles of vertices of $D'$ (inside $\HH^2$). At every cycle of vertices the sum of the angles is equal to $p\frac{2\pi}{p}=2\pi$. Hence, the group acts freely on $\HH^2$ and the quotient is a smooth surface. Moreover, we can express $\eta_{0(p-1,j)}$ from the relations, and we are left with a free group of rank $(p-1)(q-1)$ generated by the elements
\begin{gather*}
\eta_{0(i,j)} \quad ,\quad i=0,...,p-2,\; j=1,...,q-1 \;.
\end{gather*}
This completes the proof.
\end{proof}

\begin{remark}
Proposition \ref{Prop Comut Quad Group} implies that the quotient $F:=\HH^2/\Gamma_{p,q}'$ is an orientable smooth open surface, whose homotopy type is a bouquet of $(p-1)(q-1)$ circles. In fact, $F$ can be embedded in ${\mathbb S}^3$ as the interior of a Seifert surface for a (p,q)-torus knot.
\end{remark}

Since $p$ and $q$ are co-prime, all the cusps of the fundamental polygon $D'$ are equivalent under the action of $\Gamma_{p,q}'$. To show this in a simple way we shall construct another fundamental domain $D_1'$ for $\Gamma_{p,q}'$ with only one cusp. We use the fundamental domain $D_1$ for $\Gamma_{p,q}$ given in remark \ref{zab AltFundDom Gamma(p,q) D_1}. In remark \ref{Zab Gamma/Gamma' = Zpq} we observed that the image of the element $\gamma_0=(\alpha_0\beta_0)^{-1}$ generates $\Gamma_{p,q}/\Gamma_{p,q}'\cong \ZZ_{pq}$. Thus
$$
D_1' \; := \; \bigcup\limits_{k=0}^{pq-1} \gamma_0^{k}(D_1)
$$
is a fundamental polygon for $\Gamma_{p,q}'$. The only cusp of $D_1$, the point $\infty$, is the fixed point of the element $\gamma_0$. Hence the polygon $D_1'$ has only one cusp, $\infty$. The two infinite edges of $D_1'$ are identified by the element $\gamma_0^{pq} = (\alpha_0\beta_0)^{-pq}$ which belongs to $\Gamma_{p,q}'$.

\subsection{A representation of the torus knot group in $\wt{SL_2}(\RR)$}\label{Sect The Group G}

The following lemma gives a representation of the torus knot group
$G_K:=\pi_1({\mathbb S}^3\setminus K)$ as a discrete subgroup of
$\wt{SL_2}(\RR)$.

\begin{lemma}\label{Lemma CenExt(G)->CenEct(SL)}
Let $\Gamma_{p,q}$ be the triangle Fuchsian group described in (\ref{Quad group p,inf,q,inf}). The pre-image
$\wt{\Gamma}:=\wt{P}^{-1}(\Gamma_{p,q})$ in $\wt{SL_2}(\RR)$ is a discrete subgroup, isomorphic to the torus knot group presented in (\ref{TheKnotGroup}). There is a commutative diagram
\begin{gather*}
\begin{array}{ccccccccc}
1 & \lw & Z(\wt{\Gamma}) & \lw & \wt{\Gamma} & \lw & \Gamma_{p,q} & \lw & 1 \\
  & & \downarrow & & \downarrow & & \downarrow & & \\
1 & \lw & C &\lw &\wt{SL_2}(\RR) &\lw &PSL_2(\RR) &\lw &1
\end{array}
\end{gather*}
where the rows are central extensions, the leftmost downarrow is an isomorphism and the other two
- monomorphisms.
\end{lemma}

\begin{proof}
The group $\Gamma_{p,q}$ is generated by the elliptic elements $\alpha_0$ and $\beta_0$. The angles of rotation are
$\frac{2\pi}{p}$ and $\frac{2\pi}{q}$. The elements $\alpha_0^p$ and $\beta_0^q$ represent rotations by angle $2\pi$ centered at the respective fixed points. Hence, these elements correspond to the generator $c$ of $C=\pi_1(PSL_2(\RR))$. We have already identified $c$ as an element of $\wt{SL_2}(\RR)$. Let $\alpha$ and $\beta$ be the unique elements in $\wt{P}^{-1}(\alpha_0)$ and $\wt{P}^{-1}(\beta_0)$ respectively, such that $\alpha^p=\beta^q=c$. It is clear that $\alpha$ and $\beta$ generate $\wt{\Gamma}$. The relation $\alpha^p=\beta^q$ is (essentially) the only relation in $\wt{\Gamma}$, because the factorization by $C$ maps $\wt{SL_2}(\RR)$ on $PSL_2(\RR)$ and $\wt{\Gamma}$ on $\Gamma_{p,q}$. Obviously $\wt{\Gamma}$ is a discrete subgroup of $\wt{SL_2}(\RR)$. The center is $Z(\wt{\Gamma})=C$ because $\Gamma_{p,q}$ has no center. We have obtained a presentation
\begin{gather}\label{The Group G}
\wt{\Gamma} \cong \; <\alpha,\beta \; ; \; \alpha^p=\beta^q>
\end{gather}
which is exactly the presentation (\ref{TheKnotGroup}) of the torus knot group.
\end{proof}

\subsection{Properties of the discrete group $<\alpha,\beta \; ; \; \alpha^p=\beta^q>$}

In this section we discuss some properties of the group $\wt{\Gamma}$ defined by the above presentation. Most of these properties, if not all, are well-known, but we shall sketch the proofs for completeness and to be able to give the formulations suitable for our purposes. In particular, we shall see that for any natural number $r$ co-prime with $p$ and $q$, there is a normal subgroup $G_r\subset\wt{\Gamma}$ of index $r$, isomorphic to $\wt{\Gamma}$ as an abstract group. Although many of the properties of $\wt{\Gamma}$ can be deduced from the presentation, we shall use the notation and some facts resulting from Lemma \ref{Lemma CenExt(G)->CenEct(SL)}. In particular, the center $Z(\wt{\Gamma})$ coincides with the center of $\wt{SL_2}(\RR)$, i.e.
\begin{gather*}
Z(\wt{\Gamma})\; = \; <c> \quad , \quad c=\alpha^p=\beta^q \;.
\end{gather*}
We have $\wt{\Gamma}/Z(\wt{\Gamma})\cong \Gamma_{p,q}$.

Let us set some notational conventions. For $g\in \wt{\Gamma}$ we denote by $\ol{g}:=g \wt{\Gamma}'$ and $g_0=g Z(\wt{\Gamma})$ the corresponding cosets with respect to the commutator subgroup and the center of $\wt{\Gamma}$, resepctively. Also, we let $N_{\wt{\Gamma}}(g)$ denote the smallest normal subgroup of $\wt{\Gamma}$ containing $g$ (and similarly for other groups instead of $\wt{\Gamma}$).

\begin{proposition}\label{Prop properties of G}

(i) The factorgroup $H:=\wt{\Gamma}/\wt{\Gamma}'$ is infinite cyclic, generated by the projection $\ol{x}$ of the element $x := \alpha^{q_1}\beta^{p_1}$, where $pp_1+qq_1=1$.

(ii) The canonical projection $\wt{\Gamma} \lw \wt{\Gamma}/Z(\wt{\Gamma})\cong\Gamma_{p,q}$ induces an isomorphism of the commutator subgroups
$\wt{\Gamma}' \cong \Gamma_{p,q}'$.

(iii) We have $N_{\wt{\Gamma}}(x)=\wt{\Gamma}$, in other words, the group $\wt{\Gamma}$ is completely destroyed by adding the relation $x=1$.
\end{proposition}

\begin{proof} From the presentation (\ref{TheKnotGroup}) one can see that the elements of $\wt{\Gamma}/\wt{\Gamma}'$ are of the form
$\ol{\alpha}^i\ol{\beta}^j$. The equalities $\ol{x}^p=\ol{\beta}$ and $\ol{x}^q=\ol{\alpha}$ imply that $\ol{x}$ generates $\wt{\Gamma}/\wt{\Gamma}'$. To see that $\wt{\Gamma}/\wt{\Gamma}'$ is infinite, notice that it can be obtained from the free abelian group of rank two by adding only one relation $\ol{\alpha}^p\ol{\beta}^{-q}=1$. This proves (i).

To prove (ii), notice that any surjective homomorphism of groups induces a surjective homomorphism of the commutator subgroups. We have to prove injectivity. Let us take $g \in \wt{\Gamma}'$ and suppose $g_0=1$ in $\wt{\Gamma}/Z(\wt{\Gamma})$. This means $g=c^{m}$ for some $m\in \ZZ$. Hence $\ol{g}=\ol{x}^{pqm}$. On the other hand, $g \in \wt{\Gamma}'$ implies $\ol{g}=1$. Hence $m=0$ and $g=1$.

To prove (iii), in view of (i), it suffices to show that $N_{\wt{\Gamma}}(x)\supset \wt{\Gamma}'$. We have $\ol{x}^{pq} = \ol{c}$, so we can work in $\Gamma_{p,q}$ and show that $N_{\Gamma_{p,q}}(x_0)\supset \Gamma_{p,q}'$. Notice that $\Gamma_{p,q}'=N_{\Gamma_{p,q}}([\alpha_0^{q_1},\beta_0^{p_1}])$ since $q_1$ is co-prime with $p$ and $p_1$ is co-prime with $q$. Now, we have $[\alpha_0^{q_1},\beta_0^{p_1}]=x_0\alpha_0^{-q_1}x_0^{-1}\alpha_0^{q_1}\in N_{\Gamma_{p,q}}(x_0)$ and hence $N_{\Gamma_{p,q}}(x_0)\supset \Gamma_{p,q}'$.
\end{proof}

From Proposition \ref{Prop properties of G}(ii) and Proposition \ref{Prop Comut Quad Group} we
obtain the following classical result.

\begin{corollary}\label{Cons Comut_G-Free}
The commutator subgroup $\wt{\Gamma}'$, of the group $\wt{\Gamma}$, is free of rank
$(p-1)(q-1)$. Any of the following two sets of elements generate
$G'$.
\begin{gather}\label{SetGenG' xi}
\xi_{(i,j)} := [\alpha^i,\beta^j] =
\alpha^i\beta^j\alpha^{-i}\beta^{-j} \quad , \; i=1,...,p-1\;,
j=1,...,q-1  \;,
\end{gather}
\begin{gather}\label{SetGenG' eta}
\qquad  \eta_{(k,j)} := [\alpha,\alpha^k\beta^j\alpha^{-k}] =
\alpha^k[\alpha,\beta^j]\alpha^{-k} \; , \; k=0,...,p-2,\;
j=1,...,q-1 .
\end{gather}
\end{corollary}

It is a well-known fact that knot complements are aspherical. Hence knot groups do not admit elements of finite order. In our case this is immediately deduced by the fact that $\wt{\Gamma}'$ and $\wt{\Gamma}/\wt{\Gamma}'$ are free. It also follows from the presence of the faithful representation $\wt{\Gamma}\lw \wt{SL_2}(\RR)$ obtained in Lemma \ref{Lemma CenExt(G)->CenEct(SL)}, and the fact that $\wt{SL_2}(\RR)$ has no elements of finite order.\\

Now we identify a family of subgroups $G_r$ of $\wt{\Gamma}$ such that each $G_r$ is isomorphic to $\wt{\Gamma}$, parametrized by natural numbers $r$ co-prime to both $p$ and $q$. Fix one such $r$ and let $r_p,r_q,p_r,q_r$ be integers such that
$$
pp_r+rr_p=1 \quad,\quad qq_r+rr_q=1 \;.
$$

\begin{lemma}\label{Lemma G1 -> G}
Let us denote $\alpha_r:=\alpha^r$, $\beta_r:=\beta^r$, and let $G_r$ be the subgroup of $\wt{\Gamma}$ generated by $\alpha_r$ and $\beta_r$. Then $G_r$ is isomorphic to $\wt{\Gamma}$, and has the following properties.

{\rm (i)} The center $Z(G_r)$ is contained in $Z(\wt{\Gamma})$ and generated by $c_r:=c^r$.

{\rm (ii)} The commutator subgroup $G_r'$ of $G_r$ coincides with the commutator subgroup $\wt{\Gamma}'$ of $\wt{\Gamma}$.

{\rm (iii)} The factorgroup $H_r:=G_r/G_r'$ is generated by the projection $\ol{x}_r$ of the element
$x_r:=\alpha_r^{q_1}\beta_r^{p_1}$, where as earlier $pp_1+qq_1=1$. We have $N_{\wt{\Gamma}}(x_r)=N_{G_r}(x_r)=G_r$.

{\rm (iv)} $G_r$ is a normal subgroup of $\wt{\Gamma}$ and the factorgroup is cyclic of order $r$. There are isomorphisms
$$
\wt{\Gamma}/G_r  \cong H/H_r \cong Z(\wt{\Gamma})/Z(G_r) \cong \ZZ_r \;.
$$

{\rm (v)} The restriction of the projection $\wt{P}:\wt{\Gamma}\lw \wt{\Gamma}/Z(\wt{\Gamma})$ to $G_r$ is surjective, i.e.
$$
\wt{P}(G_r)=\Gamma_{p,q} \;.
$$
\end{lemma}

\begin{proof}
We obviously have $\alpha_r^p=\beta_r^q=c^r=c_r$. The group $\wt{\Gamma}$ can be obtained as a free product with amalgamation
of the free groups $<\alpha>$, $<\beta>$, with the amalgamated subgroups $<\alpha^p>\sim<\beta^q>$. The cyclic groups
$<\alpha_r>$ and $<\beta_r>$ are subgroups of $<\alpha>$ and $<\beta>$ respectively. We have $<\alpha_r>\cap<\alpha^p>=<\alpha_r^p>$ and $<\beta_r>\cap<\beta^q>=<\beta_r^q>$. Thus
\begin{gather*}
G_r \cong \; <\alpha_r,\beta_r \; ;\; \alpha_r^p=\beta_r^q> \;.
\end{gather*}
Hence $G_r$ is isomorphic to $\wt{\Gamma}$ and an isomorphism is given by
\begin{gather}\label{Isomorph G_1 G}
{\mc I}_r : \wt{\Gamma} \lw G_r \quad,\quad \alpha \lo \alpha_r \quad,\quad \beta \lo \beta_r \;.
\end{gather}
Since $c_r=\alpha_r^p=\beta_r^q$, property (i) is obvious. Moreover, we have $Z(\wt{\Gamma})/Z(G_r)\cong\ZZ_{r}$.

To prove property (ii) it is sufficient to see that the elements $\xi_{(i,j)}=[\alpha^i,\beta^j]$ belong to $G_r'$ (see corollary \ref{Cons Comut_G-Free}). We have
$$
\alpha_r^{r_p}=\alpha^{rr_p}=\alpha^{rr_p+pp_r}c^{-p_r}=\alpha c^{-p_r} \quad,\quad \beta_r^{r_q}=\beta c^{-q_r} \;.
$$
Hence $[ \alpha^i,\beta^{j} ] = [ \alpha_r^{ir_p}, \beta_r^{jr_q} ]\in G_r'$.

The proof of property (iii) is completely analogous to the proof of parts (i) and (iii) of Proposition \ref{Prop properties of G}.

To prove (iv) we observe that, since $G_r'=\wt{\Gamma}'$, there is a natural monomorphism
\begin{gather*}
\begin{array}{ccc}
H_r & \hookrightarrow & H \\
\ol{x}_r & \lo & \ol{x}^{r}
\end{array} \quad .
\end{gather*}
The quotient is $H/H_r \cong  \ZZ_r$.

The subgroup $G_r$ is normal in $\wt{\Gamma}$, because $\wt{\Gamma}'=G_r'\subset G_r$ and for any pair of elements $g\in \wt{\Gamma}$ and $g_1 \in G_r$ we have $gg_1g^{-1}=[g,g_1]g_1 \in G_r$. We obtain
$$
\wt{\Gamma}/G_r  \cong  (\wt{\Gamma}/\wt{\Gamma}')/(G_r/G_r')  \cong  \sco{<\ol{x}>/<\ol{x}_r>}  \cong  \sco{<\ol{x}>/<\ol{x}^{r}>} \cong  \ZZ_{r} .
$$
It remains to prove (v). Recall that $\wt{P}(\wt{\Gamma})=\Gamma_{p,q}=<\alpha_0,\beta_0 ; \alpha_0^p=\beta_0^q=1>\cong \ZZ_p*\ZZ_q$, where $\alpha_0=\wt{P}(\alpha)$ and $\beta_0=\wt{P}(\beta)$. We have $\wt{P}(\alpha_r)^{r_p}=\alpha_0$ and $\wt{P}(\beta_r)^{r_q}=\beta_0$. Hence $\wt{P}(G_r)=\Gamma_{p,q}$. The lemma is proved.
\end{proof}

\begin{corollary} For any natural number $r$ co-prime to both $p$ and $q$, the coset space $\wt{SL_2}(\RR)/G_r$ is weakly homotopically equivalent to the torus knot complement ${\mathbb S}^3\setminus K_{p,q}$.
\end{corollary}

Our task now becomes to identify an $r$ for which $\wt{SL_2}(\RR)/G_r$ is in fact diffeomorphic to the torus knot complement. As pointed out in the introduction, this has been done by Raymond and Vasquez, in \cite{RayVas}, using the theory of Seifert fibrations. We intend to find an explicit diffeomorphism using the results on automorphic forms proved in the next section.

\subsection{Automorphic forms for a $(p,q,\infty)$-triangle group}\label{Sect Auto-f Gamma_pq}

In this section we study automorphic forms related to the $(p,q,\infty)$-triangle group
\begin{gather*}
\Gamma_{p,q} =\; <\alpha_0,\beta_0\; ;\; \alpha_0^p=1 , \beta^q=1 > \;.
\end{gather*}
We shall refer to the fundamental domain $D_1$ for the action of $\Gamma_{p,q}$ on $\HH^2$ described in remark \ref{zab AltFundDom Gamma(p,q) D_1}. Recall that the parabolic element $\gamma_0=(\alpha_0\beta_0)^{-1}$ generates the stabilizer of the cusp $\infty$. According to formulae \ref{Javni formuli Gamma(p,q)}, we have
$$
\gamma_0(z)=z+\lambda \quad,\quad {\rm where} \quad \lambda = 2(\cos\frac{\pi}{p}+\cos\frac{\pi}{q}) \;.
$$
The preimage $\wt{\Gamma}=\wt{P}^{-1}(\Gamma_{p,q})$ in $\wt{SL_2}(\RR)$ is generated by the preimages $\alpha,\beta$ as defined in the proof of Lemma \ref{Lemma CenExt(G)->CenEct(SL)}. We want to understand the algebra of automorphic forms ${\mc A}_{\wt\Gamma}^{*,*}$. Let us interpret the definition of automorphic forms for the group in question. Let $\omega(z,dz)=f(z)dz^k$ be a $k$-form on $\HH^2$, with $k\in\QQ$ and $f(z)$ holomorphic on $\HH^2$. Let $\chi\in{\rm Hom}(\wt{\Gamma},U(1))$ be a character. Then $\omega$ is $\chi$-automorphic if

(a) $f(\alpha_0(z))\alpha'(z) = \chi(\alpha)f(z)$.

(b) $f(\beta_0(z))\beta'(z) = \chi(\beta)f(z)$.

(c) $f(z)=\sum_{n=0}^{\infty} a_ne^{2\pi i \frac{z}{\lambda}}$.

\begin{remark} Notice that a character $\chi$ of $\wt{\Gamma}$ is completely determined by its value on the element $x=\alpha^{q_1}\beta^{p_1}$. This follows directly from the facts that the kernel of any character must contain the commutator subgroup $\wt{\Gamma}'$ and that $\ol{x}$ generates $\wt{\Gamma}/\wt{\Gamma}'$.
\end{remark}

\begin{lemma}\label{Lemma N+...=k/k_o}
Let $\omega(z,dz)=f(z)dz^k$ be a nonzero $\chi$-automorphic $k$-form for $\wt{\Gamma}$. Let $n_a$, $n_b$ and $n_\infty$ denote the orders of vanishing of $f$ at $a$, $b$ and $\infty$ respectively; let $N$ denote the total number of other zeros of $f$ in the fundamental domain $D_1$, counted with multiplicities, with the identifications of the edges taken into account. Then
$$
N + n_\infty + \frac{n_a}{p} + \frac{n_b}{q} = k \frac{pq-p-q}{pq} \;.
$$
\end{lemma}

The proof is standard, uses a contour integral, and is analogous to the proof of Lemma 1 on page I-14 of \cite{Ogg}.\\

\begin{corollary}\label{Coro k(pq-p-q) integer}
In the notation of the above lemma, the number $m=k(pq-p-q)$ is an integer. The integer $m$ is divisible by $p$ (respectively, by $q$) if and only if $n_a$ (respectively, $n_b$) is. In particular, if $n_a=n_b=0$, then $k\frac{pq-p-q}{pq}$ is an integer.
\end{corollary}

Put
$$
k_o = \frac{pq}{pq-p-q} \;.
$$
(This number has some geometric meaning. Namely, the fundamental triangle $\Delta$ (see section \ref{Sect Def Triang Groups}) has Lobachevsky area $\frac{\pi}{k_o}$. This follows directly from the formula $\frac{1}{k_o}=1-\frac{1}{p}-\frac{1}{q}$.)

The following two lemmas are taken from Milnor's article \cite{MilnorBrieskornMani}. They relate the character of an automorphic form to the order of vanishing of the form at the vertices $a$ and $b$.

\begin{lemma}\label{Lemma chi(a) and chi(b)}
In the notation of Lemma \ref{Lemma N+...=k/k_o}, the values of the character $\chi$ on the generators $\alpha$ and $\beta$ of $\wt{\Gamma}$ are
$$
\chi(\alpha)=e^{2\pi i (\frac{k+n_a}{p})} \quad , \quad \chi(\beta)=e^{2\pi i (\frac{k+n_b}{q})} \;.
$$
In particular, the value of $\chi$ on the element $x$ is $\chi(x)=e^{2\pi i \frac{k}{pq}} e^{2\pi i \frac{qq_1n_a + pp_1n_b}{pq}}$.
\end{lemma}

\begin{proof} We will evaluate $\chi(\alpha)$, the case of $\beta$ being analogous. Since $\omega$ is $\chi$-automorphic, $f$ satisfies the relation
$$
f(\alpha(z))(\alpha'(z))^k = \chi(\alpha) f(z) \;.
$$
Notice that $(\alpha'(a))^k$ projects to $e^{2\pi i \frac{k}{p}}$ in $\CC^\times$. Expand $f(z)$ and $f(\alpha(z))$ as Taylor series about $a$:
$$
f(z)=\sum\limits_{n=n_a}^{\infty} f_n(z-a)^n \quad , \quad f(\alpha(z)) = \sum\limits_{n=n_a}^{\infty} f_n e^{2\pi i \frac{n}{p}} (z-a)^n \;.
$$
Now the automorphy relation can be written as
$$
\sum\limits_{n=n_a}^{\infty} e^{2\pi i \frac{k}{p}} f_n e^{2\pi i \frac{n}{p}} (z-a)^n = \sum\limits_{n=n_a}^{\infty} \chi(\alpha)f_n(z-a)^n \;.
$$
A comparison of the coefficients in front of $(z-a)^{n_a}$ yields the desired
$$
e^{2\pi i \frac{n_a+k}{p}} = \chi(\alpha) \;.
$$
The last statement in the lemma follows immediately from the expression $x=\alpha^{q_1}\beta^{p_1}$.
\end{proof}

Now, we identify a character $\chi_o$, which is of some particular importance. Set
$$
\chi_o(x) = e^{2\pi i \frac{k_o}{pq}} = e^{2\pi i \frac{1}{pq-p-q}} \;.
$$
On the generators $\alpha$ and $\beta$ we get
$$
\chi_o(\alpha)=e^{2\pi i \frac{k_o}{p}} = e^{2\pi i \frac{q}{pq-p-q}} \quad,\quad \chi_o(\beta)=e^{2\pi i \frac{k_o}{q}} = e^{2\pi i \frac{p}{pq-p-q}} \;.
$$
Clearly, we have $\chi_o^{pq-p-q}=1$.

\begin{lemma}\label{Lemma chi=chi_o^k/k_o}
In the notation of Lemma \ref{Lemma N+...=k/k_o}, assume $n_a=n_b=0$. Then $k/k_o$ is a non-negative integer and $\chi=\chi_o^{k/k_o}$.
\end{lemma}

\begin{proof} The fact that $k/k_o$ is a non-negative integer follows directly from Lemma \ref{Lemma N+...=k/k_o}. From Lemma \ref{Lemma chi(a) and chi(b)} we obtain $\chi(\alpha)=e^{2\pi i \frac{k}{p}}$ and $\chi(\beta)=e^{2\pi i \frac{k}{q}}$, whence
$$
\chi(x)=e^{2\pi i \frac{k}{pq}} = e^{2\pi i \frac{k_o}{pq}\frac{k}{k_o}} = \chi_o(x)^{k/k_o} \;.
$$
Since the characters of $\wt{\Gamma}$ are determined by their values on $x$, we get $\chi=\chi_o^{k/k_o}$.
\end{proof}

We now proceed to construct some specific automorphic forms. Recall from Proposition \ref{Prop H->Theta} that the orbifold $\Theta=\HH^2/\Gamma_{p,q}$ has underlying surface $X_{\Theta}$ isomorphic to $\CC$. The singular locus $\Sigma_\Theta$ consists of two cone points of indices $p$ and $q$ respectively: the images of the vertices $a$ and $b$ of the fundamental domain $D_1$. Thus we have a uniformizing function for $\Theta$,
$$
\theta : \HH^2 \lw \CC = X_{\Theta} \;,
$$
which can be chosen so that $\theta(a)=0$ with multiplicity $p$, $\theta(b)=1$ with multiplicity $q$. Let $\wh{\HH^2}=\HH^2\cup\{\infty\}$ and $\wh{\CC}=\CC\cup\{\infty\}=\CC\PP^1$. Then $\theta$ extends to a function $\theta:\wh{\HH^2}\lw\wh{\CC}$ with a simple pole at $\infty$.

The function $\theta$ is invariant\footnote{In many texts such a function would be called automorphic, but it does not satisfy the definition adopted here because of the pole at $\infty$.} on the orbits of $\Gamma_{p,q}$, i.e. $\theta(\gamma(z))=\theta(z)$ for all $\gamma\in\Gamma_{p,q}$. The derivative $\theta'(z)$ satisfies
$$
\theta'(\gamma(z))=(\gamma'(z))^{-1}\theta'(z) \quad , \quad \gamma\in\Gamma_{p,q} \;.
$$
Set
\begin{align*}
f_a=&\sco{\frac{(\theta')^q}{\theta(\theta-1)^{q-1}}}^{\frac{1}{pq-p-q}} \;,\\
f_b=&\sco{\frac{(\theta')^p}{\theta^{p-1}(\theta-1)}}^{\frac{1}{pq-p-q}} \;,\\
f_{\infty}=&\sco{\frac{(\theta')^{pq}}{\theta^{(p-1)q}(\theta-1)^{p(q-1)}}}^{\frac{1}{pq-p-q}} \;.
\end{align*}
It is a matter of direct verification, by counting zeros and poles in numerator and denominator, to show that $f_a,f_b,f_\infty$ are holomorphic functions on $\HH^2$ and at $\infty$. Furthermore, $f_a$ (resp. $f_b$, $f_\infty$) has a simple zero at $a$ (resp. $b$, $\infty$) and nowhere else in the closure $\wh{D_1}=\ol{D_1}\cup\{\infty\}$ of the fundamental domain of $\Gamma_{p,q}$ given in \ref{zab AltFundDom Gamma(p,q) D_1}.

\begin{lemma}\label{Lemma f_a,f_b,f_infty}
The forms
\begin{gather*}
\omega_a(z,dz)=f_a(z)dz^{\frac{q}{pq-p-q}} \quad,\quad \omega_b(z,dz)=f_b(z)dz^{\frac{p}{pq-p-q}} \quad,\\
\omega_{\infty}(z,dz)=f_{\infty}(z)dz^{\frac{pq}{pq-p-q}}
\end{gather*}
are automorphic forms for $\wt{\Gamma}$. The character of $\omega_\infty$ is $\chi_o$. The characters $\chi_a$ and $\chi_b$ of $\omega_a$ and $\omega_b$ evaluated at $x$ give
$$
\chi_a(x) = e^{2\pi i \frac{1}{p(pq-p-q)}}e^{2\pi i \frac{q_1}{p}} \quad,\quad \chi_b(x) = e^{2\pi i \frac{1}{q(pq-p-q)}}e^{2\pi i \frac{p_1}{q}} \;.
$$
The relations $\chi_a^p=\chi_b^q=\chi_o$ hold.
\end{lemma}

\begin{proof}
For $\gamma\in \wt{\Gamma}$ we have
\begin{align*}
f_a(\gamma(z)) & = \sco{\frac{(\theta'(\gamma(z)))^q}{\theta(\gamma(z))(\theta(\gamma(z))-1)^{q-1}}}^{\frac{1}{pq-p-q}} \\
 & = \sco{\frac{((\gamma'(z))^{-1}\theta'(z))^q}{\theta(z)(\theta(z)-1)^{q-1}}}^{\frac{1}{pq-p-q}} \\
 & = \chi(\gamma)(\gamma'(z))^{\frac{-q}{pq-p-q}} \sco{\frac{(\theta'(z))^q}{\theta(z)(\theta(z)-1)^{q-1}}}^{\frac{1}{pq-p-q}}\\
 & = \chi(\gamma)(\gamma'(z))^{\frac{-q}{pq-p-q}} f_a(z)
\end{align*}
where $\chi_a(\gamma)$ is a $(pq-p-q)$-th root of 1. The mapping $\gamma\lo\chi_a(\gamma)$ defines a character, and $\omega$ is a $\chi_a$-automorphic form. Since the zeros of $f_a$ are known, Lemma \ref{Lemma chi(a) and chi(b)} yields
$$
\chi_a(x) = e^{2\pi i \frac{1}{p(pq-p-q)}}e^{2\pi i \frac{q_1}{p}} \;.
$$
The situation with $\omega_b$ and $\omega_\infty$ is completely analogous. The relation between the characters follows immediately from the explicit formulae.
\end{proof}

\begin{lemma}\label{lemma f_a^p-f_b^q=f_inf}
The forms $\omega_a$, $\omega_b$, $\omega_{\infty}$ defined above satisfy a relation of the form
$$
\omega_{\infty} = c_a\omega_a^p + c_b\omega_b^q \;.
$$
for some nonzero complex numbers $c_a,c_b$.
\end{lemma}

\begin{proof} For $u,v\in\CC$ we compute
\begin{align*}
uf_a^p + vf_{\infty} & = u\sco{\frac{(\theta')^{pq}}{\theta^p(\theta-1)^{p(q-1)}}}^{\frac{1}{pq-p-q}}
                       + v\sco{\frac{(\theta')^{pq}}{\theta^{(p-1)q}(\theta-1)^{p(q-1)}}}^{\frac{1}{pq-p-q}} \\
                     & = \sco{\frac{(\theta')^{pq}}{\theta^{(p-1)q}(\theta-1)^{q}}}^{\frac{1}{pq-p-q}} \sco{u\varepsilon_1\theta(\theta-1)^{-1} + v\varepsilon_2(\theta-1)^{-1}} \\
                     & = f_b^q \sco{\frac{u\varepsilon_1\theta+v\varepsilon_2}{\theta-1}}
\end{align*}
where $\varepsilon_1$ and $\varepsilon_2$ are suitable roots of unity of order $pq-p-q$. Hence for $u=\varepsilon_1^{-1}$ and $v=-\varepsilon_2^{-1}$ we get $uf_a^p+vf_\infty=f_b^q$ which implies the statement of the lemma.
\end{proof}

\begin{theorem}\label{Theo Gen AutFormsGamma}
The forms $\omega_a$ and $\omega_b$ generate the algebra ${\mc A}_{\wt{\Gamma}}^{*,*}$.
\end{theorem}

\begin{proof}
Let $\omega(z,dz)=f(z)dz^k\in{\mc A}_{\wt{\Gamma}}^{k,\chi}$ be a nonconstant automorphic form. In view of Lemma \ref{lemma f_a^p-f_b^q=f_inf}, it is sufficient to express $f$ as a polynomial in $f_a$, $f_b$, $f_\infty$. We shall use the notation from Lemma \ref{Lemma N+...=k/k_o} for the orders of vanishing of $f$. According to corollary \ref{Coro k(pq-p-q) integer}, the number $m=k(pq-p-q)$ is an integer, which is divisible by $p$ if and only if $n_a$ is. We shall consider two cases.

{\it Case 1}: Suppose that $p$ divides $m$, say $m=ps$. Then $n_1=\frac{n_a}{p}$ is also an integer. Then the form $\omega_b^s$ is a $k$-form. Moreover, we have $\chi=\chi_b^s$. To see this we just evaluate the two characters at $x$ using Lemma \ref{Lemma chi(a) and chi(b)}:
\begin{align*}
\chi(x) & = e^{2\pi i \frac{m}{pq(pq-p-q)}}e^{2\pi i \frac{qq_1n_a+pp_1n_b}{pq}} = e^{2\pi i \frac{s}{q(pq-p-q)}}e^{2\pi i \frac{qq_1pn_1+pp_1n_b}{pq}} \\ & = e^{2\pi i \frac{s}{q(pq-p-q)}}e^{2\pi i \frac{p_1n_b}{q}} \;,\\
\chi_b^{s}(x) &= e^{2\pi i \frac{s}{q(pq-p-q)}}e^{2\pi i \frac{p_1s}{q}} \;.
\end{align*}
Thus $\chi=\chi_b^{s}$ is equivalent to $e^{2\pi i \frac{p_1n_b}{q}} = e^{2\pi i \frac{p_1s}{q}}$, which in turn holds if and only if $q$ divides $s-n_b$. The identity proven in Lemma \ref{Lemma N+...=k/k_o}, properly rewritten according to the present assumptions, becomes
$$
ps = pq N + pq n_\infty + q pn_1 + p n_b \;.
$$
The above equality implies that $q$ divides $s-n_b$. Hence $\chi=\chi_b^{s}$. We can conclude that both $\omega$ and $\omega_b^{s}$ belong to the vector space ${\mc A}_{\wt{\Gamma}}^{k,\chi}$. Hence any linear combination of these two forms belongs to ${\mc A}_{\wt{\Gamma}}^{k,\chi}$. There exists $c_1\in\CC$ such that the form $f-c_1f_b$ vanishes at $\infty$ and can be written as
$$
f-c_1f_b^s=f_{\infty}f_1 \quad,\quad f_1(z)dz^{k_1} \in {\mc A}_{\wt{\Gamma}}^{k_1,\chi_1} \;,
$$
where $\chi_1=\chi\chi_o^{-1}$ and $k_1=k-k_o=\frac{p(s-q)}{pq-p-q}$. Put $m_1=p(s-q)$. Now notice the following: the fact that ${\mc A}_{\wt{\Gamma}}^{k,\chi}$ contains cusp forms implies, via Lemma \ref{Lemma N+...=k/k_o}, that $m\geq pq$. Thus $s\geq q$. There are two possible outcomes, depending on whether $s=q$ or $s>q$. First, if $s=q$, then $k_1=0$, so that $f_1$ is a holomorphic function on $\wh{\HH^2/\Gamma_{p,q}}$ and hence constant. In this case $f=c_1f_b^s+f_1f_{\infty}$, and we are done. Second, if $s>q$, then $p$ divides $m_1$ and we can restart the procedure. Eventually $f$ is expressed as a polynomial in $f_b$ and $f_\infty$.

{\it Case 2}: Suppose now that $p$ does not divide $m$. Then $n_a> 0$. Consider $f_1=\frac{f}{f_a^{n_a}}$. Put $m_1=m-qn_a$. Then $f_1(z)dz^{\frac{m_1}{pq-p-q}}$ is $\chi\chi_a^{-n_a}$-automorphic. We have $f_1(a)\ne 0$ and hence $p$ divides $m_1$, which brings us to case 1.

This completes the proof of the theorem.
\end{proof}

\begin{corollary}\label{Coro Gen A_G}
Let $G=G_{pq-p-q}$ be the subgroup of $\wt{\Gamma}$ generated by $\alpha^{pq-p-q}$ and $\beta^{pq-p-q}$. Then ${\mc A}_G^*={\mc A}_{\wt{\Gamma}}^{*,*}$, and hence the algebra ${\mc A}_G^*$ is generated by $\omega_a$ and $\omega_b$.
\end{corollary}

\begin{proof} To prove the corollary it suffices to show that $G$ equals the kernel of each of the characters $\chi_a$ and $\chi_b$. According to Lemma \ref{Lemma G1 -> G} the group $G$ is isomorphic to $\wt{\Gamma}$ as an abstract group, and equals the smallest normal subgroup of $\wt{\Gamma}$ containing the element $x^{pq-p-q}$. Also, we have $\wt{P}(G)=\Gamma_{p,q}$. From Lemma \ref{Lemma f_a,f_b,f_infty} we know that $\chi_a^p=\chi_b^q=\chi_o$. We also know that $\chi_o(x)=e^{2\pi i\frac{1}{pq-p-q}}$ so that $\chi_o$ has order $pq-p-q$. Let $s$ be the order of $\chi_a$. Then $1=\chi_a^{ps}=\chi_o^s$, whence $pq-p-q$ divides $s$. On the other hand, we have
\begin{gather*}
\chi_a(x)=e^{2\pi i \frac{1}{p(pq-p-q)}}e^{2\pi i \frac{q_1}{p}} = e^{2\pi i \frac{1+pqq_1-pq_1-qq_1}{p(pq-p-q)}} = e^{2\pi i \frac{pp_1+qq_1+pqq_1-pq_1-qq_1}{p(pq-p-q)}} \\
= e^{2\pi i \frac{pp_1+pqq_1-pq_1}{p(pq-p-q)}} = e^{2\pi i \frac{p_1+qq_1-q_1}{pq-p-q}} \;.
\end{gather*}
Hence $\chi_a^{pq-p-q}=1$ and so $s$ divides $pq-p-q$. Thus $s=pq-p-q$. An analogous argument shows that the order of $\chi_b$ equals $pq-p-q$. This completes the proof.
\end{proof}

\section{The universal covering $\wt{SL_2}(\RR)\lw {\mathbb S}^3\setminus K_{p,q}$}\label{Sect SL/G = S-K}

In this section we show that $\wt{SL_2}/G$ is diffeomorphic to the complement ${\mathbb S}^3\setminus K_{p,q}$ of a torus knot in the 3-sphere, where $G$ is the group defined in corollary \ref{Coro Gen A_G}. Consider the map
\begin{gather*}
\begin{array}{cccc}
\Psi: & \HH^2\times\wt{\CC^\times} & \lw & \CC^2 \\
      & (z,w) & \lo & (\omega_b(z,w),\omega_a(z,w))
\end{array} \;.
\end{gather*}
Corollary \ref{Coro Gen A_G} asserts that the map $\Psi$ factors through the projection $P_G:\HH^2\times\wt{\CC^\times} \lw (\HH^2\times\wt{\CC^\times})/G$. Put $M=(\HH^2\times\wt{\CC^\times})/G$ and denote the resulting map by
\begin{gather*}
\ol{\Psi}: M \lw \CC^2 \;.
\end{gather*}
Now recall that we have $\omega_\infty=c_b\omega_b^q+c_a\omega_a^p$. The cusp form $\omega_\infty$ does not vanish on $\HH^2\times\wt{\CC^\times}$. It follows that the image of $\ol{\Psi}$ sits in the complement of the curve $V\subset\CC^2$ defined by the equation $c_bz_1^q+c_az_2^p=0$.

\begin{theorem}\label{Theo HxC/G = C^2-V}
The map $\ol{\Psi}:M\lw\CC^2\setminus V$ is biholomorphic and $\wt{\CC^\times}$-equivariant.
\end{theorem}

\begin{proof}

We shall prove that $\ol{\Psi}$ is a bijection. Since $\omega_a$ and $\omega_b$ are holomorphic functions on $M$, this will be sufficient to conclude that $\ol{\Psi}$ is biholomorphic (see \cite{BochnerMartin}, p. 179). Put $r=pq-p-q$.

First we prove the injectivity of $\ol{\Psi}$. Recall that $\omega_a$ and $\omega_b$ generate ${\mc A}_G^*$. Hence if $\Psi(z,w)=\Psi(z_0,w_0)$, then $\omega(z,w)=\omega(z_0,w_0)$ for all $\omega\in{\mc A}_G^*$. To prove injectivity it is sufficient to show that for any two points $(z,w),(z_0,w_0)\in\HH^2\times\wt{\CC^\times}$ which are not congruent under the action of $G$, there exists a $G$-automorphic form  $\omega$ such that $\omega(z_0,w_0)\ne\omega(z,w)$. We can assume that both $z$ and $z_0$ belong to the fundamental domain $D_1$ of $\Gamma_{p,q}$. Recall that $\omega_a^p$ and $\omega_b^q$ belong to ${\mc A}_G^{k_o}$. Also, $f_a^p$ vanishes at the vertex $a$ with multiplicity $p$ and has no other zeros in $D_1$. Similarly, $f_b^q$ vanishes at $b$ and nowhere else in $\ol{D_1}$. Thus there exists a unique (up to a scalar multiple) linear combination $f=c_1f_b^q+c_2f_a^p$ which vanishes at $z_0$. This $f$ does not vanish at any point in $\ol{D_1}$ not congruent to $z_0$. Now the point $z$ is either 
congruent to $z_0$, or not. If $z$ is not congruent to $z_0$, then $\omega(z,dz)=f(z)dz^{k_o}$ vanishes at $(z_0,w_0)$ but not at $(z,w)$. If $z$ is congruent to $z_0$, we can assume that $z=z_0$ and that $\omega(z,dz)$ does not vanish at $(z_0,w_0)$ nor at $(z_0,w)$. However, an elementary direct computation shows that, if $\Psi(z_0,w_0)=\Psi(z_0,w)$, then $w_0=w(c'(z_0))^{jr}$ for some integer $j$, so that $(z_0,w_0)$ and $(z_0,w)$ are congruent under (the center of) $G$.

Now we prove the surjectivity of $\ol{\Psi}$. Start by noticing that both $\HH^2\times\wt{\CC^\times}$ and $\CC^2$ carry $\wt{\CC^\times}$ actions with respect to which $\Psi$ is equivariant. These actions are given by
\begin{gather*}
\begin{array}{rcl}
\wt{\CC^\times}\times(\HH^2\times\wt{\CC^\times}) & \lw & \HH^2\times\wt{\CC^\times} \\
 \lambda\cdot(z,w) & = & (z,\lambda w)
\end{array} \quad,\quad
\begin{array}{rcl}
\wt{\CC^\times}\times\CC^2 & \lw & \CC^2 \\
 \lambda\cdot(z_1,z_2) & = & (\lambda^{\frac{p}{r}} z_1,\lambda^{\frac{q}{r}} z_2)
\end{array} \;.
\end{gather*}
(Recall that $\wt{\CC^{\times}}$ acts on $\CC$ via the projection to $\CC^\times$.) The second action was already introduced in section \ref{Sect Knots}. Let us recall and re-examine the orbit structure. It is easy to see that each orbit closure has the form
$$
\{(z_1,z_2)\in\CC^2 : c_1z_1^q+c_2z_2^p=0\}
$$
for some suitable pair of complex numbers $(c_1,c_2)$. Conversely, each pair of complex numbers $(c_1,c_2)$ corresponds to an orbit closure. We want to show that all orbits, except the ones corresponding to pairs proportional to $(c_b,c_a)$, are contained in the image of $\Psi$. Let $(c_1,c_2)$ be a pair not proportional to $(c_b,c_a)$. Set
$$
\omega = c_1\omega_b^q + c_2\omega_a^p \;.
$$
Then $\omega$ is a form in ${\mc A}_G^{k_o}$ and we have $\omega(z,dz)=f(z)dz^{k_o}$ with $f=c_1f_b^q+c_2f_a^p$. The choice of $c_1,c_2$ implies that $f$ is not a scalar multiple of $f_\infty$. Hence $f$ vanishes somewhere inside $\HH^2$, say $f(z_0)=0$. Then $\omega(z_0,dz)=0$ and we have
$$
\Psi(\{z_0\}\times\wt{\CC^\times}) = \{(z_1,z_2)\in\CC^2\setminus(0,0) : c_1z_1^q+c_2z_2^p=0\} \;.
$$
Thus the image of $\Psi$ equals $\CC^2\setminus V$ which implies the surjectivity of $\ol{\Psi}$. The $\wt{\CC^\times}$-equivariance $\ol{\Psi}$ follows from the $\wt{\CC^\times}$-equivariance of $\Psi$ and the fact that on $\HH^2\times\wt{\CC^\times}$ the $\wt{\CC^\times}$-action and the $G$-action commute.
\end{proof}

\begin{corollary}\label{Coro S-K = wt(SL)/G}
The coset space $\wt{SL_2}(\RR)/G$ is diffeomorphic to the complement ${\mathbb S}^3\setminus K$ of a $(p,q)$-torus knot in the 3-sphere.
\end{corollary}

\begin{proof}
Recall that we have identified $\wt{SL_2}(\RR)$ with its orbit in $\HH^2\times\wt{\CC^\times}$ through the point $(i,1)$. Restrict the two $\wt{\CC^\times}$ actions from the above proof to $\RR_+$ actions, i.e. consider only real positive $\lambda$. Then each $\RR_+$ orbit in $\HH^2\times\wt{\CC^\times}$ intersects each orbit of $\wt{SL_2}(\RR)$ at exactly one point. On the other hand, each $\RR_+$ orbit in $\CC^2$ intersects the three sphere ${\mathbb S}^3=\{\abs{z_1}^2+\abs{z_2}^2=1\}$ at exactly one point. The curve $V$ and its complement $\CC^2\setminus V$ are invariant under the $\RR_+$ action. The intersection $K={\mathbb S}^3\cap V$ is a $(p,q)$-torus knot. Thus each $\RR_+$ orbit in $\CC^2\setminus V$ intersects both $\Psi(\wt{SL_2}(\RR))$ and ${\mathbb S}^3\setminus K$, each at a single point.

Define a map $\rho:\Psi(\wt{SL_2}(\RR))\lw {\mathbb S}^3\setminus K$ sending a point of $\Psi(\wt{SL_2}(\RR))$ to the unique point in ${\mathbb S}^3\setminus K$ which belongs to the same $\RR_+$ orbit as $(z_1,z_2)$. The $\RR_+$ orbit through a point $(z_1,z_2)\in\Psi(\wt{SL_2}(\RR))$ has the form $\{(\lambda^{\frac{p}{r}}z_1,\lambda^{\frac{q}{r}}z_2):\lambda\in\RR_+\}$. The function $F(\lambda,z_1,z_2)=\norm{(\lambda^{\frac{p}{r}}z_1,\lambda^{\frac{q}{r}}z_2)}^2$ takes all positive values, and has nonvanishing partial derivative $\frac{\pl F}{\pl\lambda}$ at all points. The implicit function theorem implies that the value of $\lambda$ for which $F(\lambda,z_1,z_2)=1$ depends smoothly on $(z_1,z_2)$. We can conclude that the map $\rho$ is a diffeomorphism.
\end{proof}

\begin{remark} The above corollary shows that the torus knot complement is a homogeneous manifold, i.e. admits a transitive group action. On the other hand, $\wt{SL_2}(\RR)$ with its left invariant metric is a model for one of the eight 3-dimensional geometries, as discussed in section \ref{Sect_GeoSrt} and remark \ref{Zab SL_2 Geometry}. This metric is not right invariant. The coset $\wt{SL_2}(\RR)/G$ is obtained by letting $G$ act on $\wt{SL_2}(\RR)$ on the left. Thus the coset, and consequently ${\mathbb S}^3\setminus K$, inherits a locally homogeneous metric. However, the latter metric is not homogeneous, i.e. ${\mathbb S}^3\setminus K$ does not admit a transitive isometry group. There is only a 1-parameter isometric action on the torus knot complement, induced by the right action of the subgroup $\wt{SO_2}(\RR)$ on $\wt{SL_2}(\RR)$, and providing the Seifert fibration structure.
\end{remark}

\section{A knot in a lens space}\label{Sect A Lens Space}

We saw, in corollary \ref{Coro S-K = wt(SL)/G}, that the complement ${\mathbb S}^3\setminus K_{p,q}$ of a $(p,q)$-torus knot in the 3-sphere is diffeomorphic to the coset space $\wt{SL_2}(\RR)/G$ where $G$ is a certain subgroup of the preimage $\wt{\Gamma}=\wt{P}^{-1}(\Gamma_{p,q})$ of a $(p,q,\infty)$-triangle group $\Gamma_{p,q}\subset PSL_2(\RR)$. The index of $G$ in $\wt{\Gamma}$ is $r=pq-p-q$. The only case when $G=\wt{\Gamma}$ is $p=2$, $q=3$, which gives $\Gamma_{p,q}=PSL_2(\ZZ)$. In this case we get ${\mathbb S}^3\setminus K_{2,3} \cong PSL_2(\RR)/PSL_2(\ZZ)$. It is natural to ask whether the coset space $\wt{SL_2}(\RR)/\wt{\Gamma}=PSL_2(\RR)/\Gamma_{p,q}$ can be identified with some other space, perhaps related to the knot complement ${\mathbb S}^3\setminus K_{p,q}$. The answer is not hard to find, and we present it in this section.

In Theorem \ref{Theo HxC/G = C^2-V} and its proof we have constructed a $\wt{\CC^\times}$-equivariant biholomorphic map $\ol{\Psi}:(\HH^2\times\wt{\CC^\times})/G \lw \CC^2\setminus V$. Recall that the action of the center $C$ of $\wt{SL_2}(\RR)$ on $\HH^2\times\wt{\CC^\times}$ coincides with the action of the subgroup of $\wt{\CC^{\times}}$ which is the preimage of $1\in\CC^\times$ under the universal covering map. Now observe that together $G$ and $C$ generate exactly $\wt{\Gamma}$, so that
$$
((\HH^2\times\wt{\CC^\times})/G)/C = (\HH^2\times\wt{\CC^\times})/\wt{\Gamma} \;.
$$
Moreover, since $C\cap G = <c^r>$, the action of $C$ on $(\HH^2\times\wt{\CC^\times})/G$ reduces to a cyclic action of order $r$. From the construction in the proof of Theorem \ref{Theo HxC/G = C^2-V} it follows that this cyclic action is transmitted to an action on $\CC^2\setminus V$ generated by the linear transformation
\begin{gather*}
h_{\frac{1}{r}}=\begin{pmatrix} e^{2\pi i \frac{p}{r}} & 0 \\
                                0 & e^{2\pi i \frac{q}{r}}
                \end{pmatrix} \;.
\end{gather*}
Clearly the latter action (as well as the $\wt{\CC^\times}$-action) extends from $\CC^2\setminus V$ to $\CC^2$, the curve $V$ being an orbit closure. The action of $<h_{\frac{1}{r}}>$ on $\CC^2\setminus {\bf o}$ is free. The quotient $\Sigma=\CC^2/<h_{\frac{1}{r}}>$ is a singular complex surface. Let $W\subset \Sigma$ denote the image of $V$, which is a singular complex curve in $\Sigma$. Clearly
$$
(\CC^{2}\setminus V)/<h_{\frac{1}{r}}> = \Sigma\setminus W \;.
$$
Thus we have a $\wt{\CC^\times}$-equivariant biholomorphic map between $(\HH^2\times\wt{\CC^\times})/\wt{\Gamma}$ and $\Sigma\setminus W$. Now recall that we have identified $\HH^2\times\wt{\CC^\times}$ with the universal cover of the bundle $T'\HH^2$ of nonzero tangent vectors to $\HH^2$. Since $\wt{\Gamma}=\wt{P}^{-1}(\Gamma_{p,q})$, we have $(\HH^2\times\wt{\CC^\times})/\wt{\Gamma}\cong T'\HH^2/\Gamma_{p,q}$. To summarize, we have obtained the following.

\begin{theorem}\label{Theo HxC/Gamma = Sigma-W}
The quotient $T'\HH^2/\Gamma_{p,q}$ is $\CC^{\times}$-equivariantly biholomorphic to the complex surface $\Sigma\setminus W$.
\end{theorem}

Now let us turn our attention to the 3-manifolds. Recall that $PSL_2(\RR)$ acts simply transitively on the unit tangent bundle $U\HH^2$. Hence $PSL_2(\RR)/\Gamma_{p,q}\cong U\HH^2/\Gamma_{p,q}\subset T'\HH^2/\Gamma_{p,q}$. The above theorem implies that $PSL_2(\RR)/\Gamma_{p,q}$ embeds in $\Sigma\setminus W$ in a way that it intersects each $\RR_+$ orbit exactly once. On the other hand, consider the image of ${\mathbb S}^3\setminus K$ in $\Sigma\setminus W$. Let ${\mc L}$ and ${\mc K}$ denote respectively the images of ${\mathbb S}^3$ and $K$ in $\Sigma$. Notice that ${\mc L}$ is a lens space: the quotient of ${\mathbb S}^3$ under the free action of the cyclic group generated by $h_{\frac{1}{r}}$. In the notation of \cite{Scott} we have ${\mc L}=L(r,p(q_1-p_1+pp_1))$. Clearly ${\mc K}={\mc L}\cap W$ and ${\mc L}\setminus {\mc K}$ is the image of ${\mathbb S}^3\setminus K$ in $\Sigma$. Observe that ${\mc L}\setminus {\mc K}$ intersects each $\RR_+$ orbit in $\Sigma\setminus W$ exactly once. Thus we can use 
the $\RR_+$ action to obtain a diffeomorphism between ${\mc L}\setminus {\mc K}$ and the image of $PSL_2(\RR)/\Gamma_{p,q}$, as in the proof of corollary \ref{Coro S-K = wt(SL)/G}. Let us record this fact.

\begin{corollary}\label{Coro L-K=PSL/Gamma}
The coset space $PSL_2(\RR)/\Gamma_{p,q}$ is diffeomorphic to the knot complement in a lens space ${\mc L}\setminus {\mc K}$.
\end{corollary}

\noindent{\bf Acknowledgement:} I am grateful to my father, Vasil Tsanov, for supporting me by all means throughout the years. The opportunity to learn from him and discuss mathematics with him is essential in many ways, but it was especially so at the early stages of this project. The project was initiated a few years ago and went through several stages before the final redaction and publication. During this time I have been supported by the Institute of Mathematics and Informatics at the Bulgarian Academy of Sciences, and Queen's University. The final revision was done at Ruhr-Universit\"at Bochum, with the support of the SFB/TR12 grant. I have benefitted from the support of and conversations with Petko Nikolov, Georgi Ganchev, Ivan Dimitrov, Ram Murty, Noriko Yui, to whom I thank sincerely.

\vspace{0.5cm}

\noindent\textsc{\small Fakult\"at f\"ur Mathematik, Ruhr-Universit\"at Bochum, Raum NA 4/76, Bochum 44780, Deutschland.}\\
{\small {\it Email}: valdemar.tsanov@gmail.com}

\end{document}